\documentclass[11pt]{amsart}
\usepackage{extsizes}
\usepackage{fullpage}
\usepackage{amsfonts,xcolor,url,verbatim}
\usepackage{amssymb}
\usepackage{xcolor}
\usepackage{dsfont}

\usepackage{todonotes}
\usepackage{hyperref}

\usepackage[backend=bibtex, sorting=none, bibstyle=alphabetic, citestyle=alphabetic, sorting=nyt, maxbibnames=99, giveninits=true, isbn=false, url=false, doi=false, eprint=false]{biblatex}  %reference manager
\renewbibmacro{in:}{}
\bibliography{references.bib}
\usepackage{amssymb, calrsfs, graphics, graphicx, enumerate, enumitem, url, xcolor, hyperref, mathrsfs, listings, comment}
\usepackage{tabularx}
\usepackage[ruled, lined, linesnumbered, longend]{algorithm2e}
\usepackage[justification=centering]{caption}
\newtheorem{theorem}{Theorem}[section]
\newtheorem{lemma}[theorem]{Lemma}

\newtheorem{proposition}[theorem]{Proposition}
\newtheorem{conjecture}[theorem]{Conjecture}
\newtheorem{corollary}[theorem]{Corollary}
\newtheorem{obs}[theorem]{Observation}
\numberwithin{equation}{section}

\theoremstyle{remark}
\newtheorem*{remark}{Remark}

\lstdefinestyle{CStyle}{
    basicstyle=\footnotesize,
    breakatwhitespace=false,         
    breaklines=true,                 
    captionpos=b,                    
    keepspaces=true,                 
    numbers=left,                    
    numbersep=5pt,                  
    showspaces=false,                
    showstringspaces=false,
    showtabs=false,                  
    tabsize=2,
    language=C
}

\providecommand{\abs}[1]{\left\vert #1 \right\vert}

\title[Improved estimates for the argument and zero-counting function of $\zeta(s)$]{Improved estimates for the argument and zero-counting function of the Riemann zeta-function\\
\textnormal{(With an appendix by Andrew Fiori)}}

\author[C. Bellotti]{Chiara Bellotti}
\address[Chiara  Bellotti]{School of Science\\
The University of New South Wales, Canberra, Australia}
\email{c.bellotti@unsw.edu.au}

\author[P.-J. Wong]{Peng-Jie Wong}
\address[Peng-Jie Wong]{National Sun Yat-Sen University\\Department of Applied Mathematics\\
Kaohsiung City, Taiwan}
\email{pjwong@math.nsysu.edu.tw}

\keywords{Riemann zeta-function, zero-density estimates, explicit estimates}
\subjclass[2020]{Primary 11M06, 11M26. Secondary 11Y35}

\thanks{P.J.W. is currently supported by the NSTC grant 111-2115-M-110-005-MY3. C.B. is partially supported by the AustMS WIMSIG Cheryl E. Praeger Travel Award. All the computation has been made using the Python program available on GitHub at \cite{Bellotti_argument_zeta}.}

\begin{document}

\begin{abstract}
In this article, we improve the recent work of Hasanalizade, Shen, and Wong by establishing 
 \[
    \left| N (T)  - \frac{T}{   2 \pi} \log \left( \frac{T}{2\pi e}\right)  \right|\le 0.10076\log T+0.24460\log\log T+8.08344,
    \]
for every $T\ge e$, where $N(T)$ is the number of non-trivial zeros $\rho=\beta+i\gamma$, with $0<\gamma \le T$,  of the Riemann zeta-function $\zeta(s)$. The main source of improvement comes from implementing new subconvexity bounds for $\zeta(\sigma+it)$ on some $\sigma_k$-lines inside the critical strip.
\end{abstract}

\maketitle

\section{Introduction}

The main objective of this article is to improve the known estimates for the number of non-trivial zeros $\rho=\beta+i\gamma$, with $0<\gamma \le T$, of the Riemann zeta-function $\zeta(s)$. More precisely, we shall study
\[
N(T) = \# \{ \rho \in  \Bbb{C}  \mid  \zeta(\rho) =0,\  0  <\beta <  1, \ 0<\gamma \leq T\}
\]
for $T> 0$. The first explicit bound for $N(T)$ was obtained by von Mangoldt \cite{vMo05}, and it has been studied and improved since then. It shall be worthwhile mentioning that explicit bounds for $N(T)$ allow one to estimate certain sums over zeros of $\zeta(s)$ and lead to the explicit bounds for the prime-counting functions $\pi(x)$ and $\psi(x)$ (see, e.g., \cite{FK15}). 

Our main theorem is the following estimate that improves the recent work of Hasanalizade, Shen, and Wong \cite{HASANALIZADE2022219}:

\begin{theorem}\label{th1}
For every $T\ge e$, we have
 \[
    \left| N (T)  - \frac{T}{   2 \pi} \log \left( \frac{T}{2\pi e}\right)  \right|\le 0.10076\log T+0.24460\log\log T+8.08344
    \]
    and
     \[
    \left| N (T)  - \frac{T}{   2 \pi} \log \left( \frac{T}{2\pi e}\right)  \right|\le 0.11200\log T+0.12567\log\log T+3.77417,
    \]
 where the first estimate becomes sharper than the second one for $T\ge \exp(447.981)$.
 \end{theorem}

For the convenience of the reader, writing 
\begin{equation}\label{his-bound}
\left| N (T)  - \frac{T}{   2 \pi} \log \left( \frac{T}{2\pi e}\right)  \right|
 \le   C_1 \log T  + C_2   \log\log T  + C_3,
\end{equation}
we summarise the advances that have been made in Table \ref{table0} below.\footnote{Regrettably, as pointed out in \cite{HASANALIZADE2022219}, there is an error in \cite{Tr14-2}. Also, there is a typo in \cite{HASANALIZADE2022219} regarding the constant $C_3$. The value $9.3675$ is now substituted with $9.4925$.} 
\begin{table}[htbp] 
\centering
\begin{tabular}{ |c||c|c|c|c|   } 
 \hline  
  &  $C_1$ & $C_2$ & $C_3$ & $T_0$ \\ 
 \hline
 von Mangoldt \cite{vMo05} (1905) & 0.4320   & 1.9167  &  13.0788  &  28.5580     \\ 
 \hline
 Grossmann \cite{Gr13} (1913) & 0.2907   & 1.7862 &  7.0120 &  50  \\ 
 \hline 
 Backlund \cite{Ba18} (1918)  & 0.1370   & 0.4430  & 5.2250 & 200 \\ 
 \hline 
 Rosser \cite{Ro41} (1941)  & 0.1370   & 0.4430  & 2.4630 & 2\\ 
 \hline  
 Trudgian \cite{Trudgian20121053} (2012) &  0.1700& 0& 2.8730 & $e$\\
 \hline
 Trudgian \cite{Tr14-2} (2014)   & 0.1120   & 0.2780  & 3.3850 & $e$\\ 
 \hline 
 Platt--Trudgian \cite{PLATT2015842} (2015) & 0.1100 & 0.2900 & 3.165 & $e$\\
 \hline
Hasanalizade, Shen, and Wong \cite{HASANALIZADE2022219} (2022)   & 0.1038   & 0.2573  & 9.4925 & $e$\\ 
 \hline  
 Bellotti--Wong (2025) & 0.10076 & 0.24460 & 8.08344 & $e$\\
 \hline
\end{tabular}
   \caption{Previous explicit bounds for $N(T)$ in \eqref{his-bound}}\label{table0} 
\end{table}

The second estimate in Theorem \ref{th1} is sharper than Trudgian's bound \cite{Tr14-2} for $T>387899$.\footnote{This, together with Platt's computation of $S(T)$ as stated in \eqref{plattbound}, assures that Trudgian's bound \cite{Tr14-2} and all the explicit results relying on it remain valid.}  We also note that similar to  \cite{HASANALIZADE2022219}, we derive Theorem \ref{th1} from the following general result.

\begin{theorem}\label{generalth1}
Let $c, r, \eta$ be positive real numbers satisfying
$$
-\frac{1}{2}<c-r<1-c<-\eta<\frac{1}{4} \leq \delta:=2 c-\sigma_1-\frac{1}{2}<\frac{1}{2}<1+\eta<\sigma_1:=c+\frac{(c-1 / 2)^2}{r}<c+r
$$
and $\theta_{1+\eta} \leq 2.1$, where $\theta_y$ is defined in \eqref{deftheta}. Let $c_1, c_2, t_0$ be defined in \eqref{notation1line}, and let $k_1,k_2, k_3,t_1$ be defined in \eqref{notation1/2}. Let $n$ denote the largest $k$ that we consider when applying \eqref{yangintbound}. 
Let $T_0 \geq e$ be fixed. Then for any $T \geq T_0$, we have
$$
\left|N(T)-\frac{T}{2 \pi} \log \left(\frac{T}{2 \pi e}\right)\right| \leq C_1 \log T+\min\{C_2 \log \log T+C_3,C'_2 \log \log T+C'_3\}
$$
where the constants $C_1=C_1\left(c, r, \eta, k_2,n\right)$, $C_2=C_2\left(c, r, \eta , c_2, k_3,n\right)$, $C'_2=C'_2\left(c, r, \eta , c_2, k_3,n\right)$, $C_3=C_3(c, r, \eta , c_1, c_2, t_0, k_1, k_2, k_3, t_1,n,T_0)$, $C'_3=C'_3(c, r, \eta , c_1, c_2, t_0, k_1, k_2, k_3, t_1,n,T_0)$ are defined in \eqref{C1}, \eqref{C2}, \eqref{C2p}, \eqref{C3}, \eqref{C3p} respectively; some admissible values are recorded in Table \ref{table1}.
\end{theorem}

\begin{remark} 
The proof of Theorem \ref{generalth1} has its roots in the works of  \cite{bennett_counting_2021,Tr14-2,HASANALIZADE2022219} and flourished with the following new inputs.\\
\noindent (i) We used various improved bounds for $\zeta(s)$, including the estimates of Yang \cite{YANG2024128124} inside the critical strip. To implement these bounds, we further considered ``$n$-splitting'' of $[\frac{1}{2},1]$ in Section \ref{nsplitting} ($n=5$ gives nearly-optimal result for $C_1$). This is one of our key observations and leads us to refined estimates for $f_N$ and $F_{c,r}$.

\noindent (ii)  Instead of just using the trivial bound 
$    |\zeta(\sigma+iT)|\ge \frac{\zeta(2\sigma)}{\zeta(\sigma)}
$
for $\sigma>1$, we provided an alternative method by applying a non-trivial estimate obtained by Leong \cite{leong2024explicitestimateslogarithmicderivative} recently. In fact, the constants $C_2$ and $C_3$ in Theorem \ref{generalth1} are calculated via the trivial bound, while $C_2'$ and $C_3'$ are obtained by Leong's bound. As may be noticed, Leong's bound allows one to obtain a relatively small $C_3'$ at the expense of a larger $C_2$, which may be useful for certain applications. In addition, throughout our discussion, we shall treat the constant obtained by Leong as a parameter $b$, (see  Lemma \ref{newprop3.1} and  \eqref{alt-bd-NT}). This would allow one to update the present results with any improvement for $b$ in the lower bound \eqref{zeta-lower}.

\noindent (iii) We directly estimate $\left|N(T)-\frac{T}{2\pi} \log \left(\frac{T}{2 \pi e}\right)\right|$, which is useful for most applications. Also, doing this reduces $C_3$ in the previous estimate of Hasanalizade--Shen--Wong \cite{HASANALIZADE2022219} by roughly $\frac{1}{8}$.
\end{remark}

We shall remark that the constant $C_1=0.10076$ in Theorem \ref{th1} is the nearly-optimal value that one can obtain from Theorem \ref{generalth1} with the current knowledge on explicit estimates for $\zeta(s)$. Also, it is possible to have smaller values of $C_2,C'_2$ or of $C_3,C'_3$, at the expense of larger values for the other constants. Here, we provide two bounds with the smallest admissible $C_2,C'_2$ and $C_3,C'_3$, respectively. (Particularly, the first estimate below is always sharper than Rosser's bound \cite{Ro41} for $T\ge e$.)

\begin{corollary}\label{cor1.3}
     The following estimates hold for $T\ge e$:
    \[
    \left| N (T)  - \frac{T}{   2 \pi} \log \left( \frac{T}{2\pi e}\right)  \right|\le 0.12355\log T+\min\{0.06782\log\log T+6.25796,0.97933 \log\log T+2.05854\}
    \]
and
    \[
    \left| N (T)  - \frac{T}{   2 \pi} \log \left( \frac{T}{2\pi e}\right)  \right|\le 0.16732\log T+\min\{0.17266\log\log T+1.96334, 1.61679\log\log T+1.40271\}.
    \] 
\end{corollary}
We note that, in the first estimate of Corollary \ref{cor1.3}, for $T\ge \exp(100.193)$, the $\min$ will take the first quantity; in the second estimate, first quantity in the $\min$ is superior  for $T\ge 4.37$.

Another important consequence of Theorem \ref{generalth1} is the following general bound for the argument $\pi S(T)$ of the Riemann zeta-function along the $\frac{1}{2}$-line. (Recall that $S(T)=\frac{1}{\pi} \Delta_L \arg \zeta(s)$, where $L$ is the path formed by the straight line from $2$ to $2+i T$ and then to $\frac{1}{2}+i T$).

 \begin{theorem}\label{th2}
    Under the same assumptions and notation of Theorem \ref{generalth1}, the following estimate holds for every $T\ge T_0$, with $T_0\ge e$ fixed:
    $$
|S(T)| \leq C_1 \log T+\min\{C_2 \log \log T+\Tilde{C_3},C'_2 \log \log T+\Tilde{C'_3}\}
$$
where the constants $C_1=C_1\left(c, r, \eta, k_2,n\right)$, $C_2=C_2\left(c, r, \eta , c_2, k_3,n\right)$, $C'_2=C'_2\left(c, r, \eta , c_2, k_3,n\right)$, $\Tilde{C_3}=\Tilde{C_3}(c, r, \eta , c_1, c_2, t_0, k_1, k_2, k_3, t_1,n,T_0)$, $\Tilde{C'_3}=\Tilde{C'_3}(c, r, \eta , c_1, c_2, t_0, k_1, k_2, k_3, t_1,n,T_0)$ are defined in \eqref{C1}, \eqref{C2}, \eqref{C2p}, \eqref{tildeC3}, \eqref{TildeC3p} respectively; some admissible values are recorded in Table \ref{table1}.
 \end{theorem}
 A sharper bound for $S(T)$, up to a certain given height, can be obtained by using the database of non-trivial zeros of $\zeta(s)$ computed by Platt and made available at \cite{LMFDB}. Indeed, one has
\begin{equation}\label{plattbound}
 |S(T)| \leq 2.5167   
\end{equation}
for $0 \leq T \leq 30610046000.$ Hence, by Theorem \ref{th2} (with $T_0=30610046000$ and Table \ref{table1}) and \eqref{plattbound}, we deduce the following explicit estimate for $S(T)$.
\begin{corollary}\label{corst}
     The following estimate holds for $T\ge e$:
    \[
    \left|S(T)  \right|\le 0.10076\log T+\min\{0.24460\log\log T+7.20844, 1.68845\log\log T+1.50956\},
    \]
    where the $\min$ will take the first quantity for $T\ge\exp(51.78)$.
\end{corollary}

Two final consequences of Theorems \ref{generalth1} and \ref{th2} are bounds for the number of zeros in a unit interval and a short interval. These types of estimates have several applications in number theory, such as providing an effective disproof of the Mertens conjecture \cite{AST_1987__147-148__325_0,rozmarynowycz202}, improving the error term in the explicit version of the Riemann–von Mangoldt formula \cite{dudekcubes}, and consequently obtaining improvements related to primes between consecutive cubes and consecutive powers \cite{michaela2024errortermexplicitformula}. More generally, these two types of estimates are useful for any problems that require an estimate for the sum over the zeros of $\zeta(s)$ restricted to a certain range.
\begin{corollary}\label{cor1}
    Let $C_1,C_2,C'_2,\Tilde{C}_3,\Tilde{C'}_3$ be the constants defined in Theorems \ref{generalth1} and \ref{th2}. 
    If $T+1> 30610046000$, then we have
\begin{equation*}
    \begin{aligned}
         N(T+1)-N(T)
        \le \left(\frac{1}{2\pi}+2C_1\right)\log T+\min\{2C_2\log\log T+\mathcal{C}_3,2C'_2\log\log T+\mathcal{C'}_3\}+\frac{1}{25T},
    \end{aligned}
\end{equation*}
where
\begin{align*}
    \mathcal{C}_3=2\Tilde{C_3}+\frac{3}{4\pi}-\frac{1}{2\pi}\log(2\pi e), \quad
    \mathcal{C'}_3=2\Tilde{C'_3}+\frac{3}{4\pi}-\frac{1}{2\pi}\log(2\pi e),
\end{align*}
and
 \begin{align*}
      &\left(\frac{1}{\pi}-2.000001C_1\right)\log T-\min\{2C_2\log\log T +\mathcal{E},2C'_2\log\log T +\mathcal{E'}\}-\frac{1}{25(T-1)}\\
      &\le N(T+1)-N(T-1)\\
        &\le \left(\frac{1}{\pi}+2C_1\right)\log T+\min\{2C_2\log\log T+\mathcal{D}_3,2C'_2\log\log T+\mathcal{D'}_3\}+\frac{1}{25(T-1)},
 \end{align*}
where
\begin{align*}
    \mathcal{D}_3=2\Tilde{C_3}+\frac{\log 3-\log(2\pi e)}{\pi},\quad
    \mathcal{D'}_3=2\Tilde{C'_3}+\frac{\log 3-\log(2\pi e)}{\pi},
\end{align*}
    \begin{align*}
    \mathcal{E}=2\Tilde{C_3}+\frac{1}{\pi}+\frac{\log(3/4)}{2\pi}-\frac{1}{\pi}\log(2\pi e),\quad
\mathcal{E'}=2\Tilde{C^{\prime}_3}+\frac{1}{\pi}+\frac{\log(3/4)}{2\pi}-\frac{1}{\pi}\log(2\pi e).
\end{align*}
Consequently, for $T+1> 30610046000$,
 \begin{equation*}
    \begin{aligned}
        & N(T+1)-N(T)\\&\le \left(\frac{1}{2\pi}+0.20152\right)\log T+\min\{0.4892\log\log T+14.2040,3.3769\log\log T+2.8062\}+\frac{1}{25T}
    \end{aligned}
\end{equation*}
and
 \begin{equation*}
    \begin{aligned}
        &\left(\frac{1}{\pi}-0.201521\right)\log T-\min\{0.4892\log\log T +13.7861,3.3769\log\log T +2.3884\}-\frac{1}{25(T-1)}\\
        &\le N(T+1)-N(T-1)\\&\le \left(\frac{1}{\pi}+0.20152\right)\log T+\min\{0.4892\log\log T+13.8633,3.3769\log\log T+2.4655\}+\frac{1}{25(T-1)}.
    \end{aligned}
\end{equation*}
In addition, for  $3\le T+1\le 30610046000$,  we have
\begin{equation*}
   \frac{1}{2\pi}\log T-5.32592-\frac{1}{25T}\le N(T+1)-N(T)\le  \frac{1}{2\pi}\log T+4.8405,
\end{equation*}
and
\[
\frac{1}{\pi}\log T-5.66421-\frac{1}{25(T-1)}\le N(T+1)-N(T-1)\le  \frac{1}{\pi}\log T+4.4798
\]
\end{corollary}
We note that for the case that $T+1> 30610046000$ in Corollary \ref{cor1}, if  $T\ge \exp(51.79)$, then the $\min$ will take the first quantity in both consequential estimates.

\begin{remark}
  We shall remark that the optimal constants that one could hope to get in front of the main terms in the upper bounds of $|N(T+1)-N(T)|$ and $|N(T+1)-N(T-1)|$ in Corollary \ref{cor1} are $\frac{1}{2\pi}$ and $\frac{1}{\pi}$, respectively, which will be shown in Section \ref{proofcor12}. Furthermore, we omit the lower bound for $N(T+1)-N(T)$ when $T+1\ge 30610046000$. Indeed, following the method used in Section \ref{proofcor12} for $N(T+1)-N(T-1)$ results in a negative lower bound for $N(T+1)-N(T)$, which is worse than the trivial non-negativity of $N(T+1)-N(T)$.
\end{remark}

\subsection*{Some recent computation {\normalfont (Private communication with David Platt, June  2025)}}\label{Plattcomp}

Several authors have reported computational results for $S(T)$. The largest values we are aware of were found by Bober and Hiary \cite{Boberhiary} including $S(T)=3.3455$ near the zeros of $\zeta$ at $$
\frac{1}{2}
+7\,757\,304\,990\,367\,861\,417\,150\,213\,053.638\,6 i.
$$
To find this, the authors exploited a new algorithm to compute $\zeta$ in short ranges and some clever heuristics to decide where to look. The search was by no means exhaustive to this (or indeed any) height.

To be able to claim a bound on $|S(T)|$ to some height $T$. It is necessary to know where the non-trivial zeros of $\zeta$ are up to that height and to some reasonable precision. The most exhaustive search to date was reported by Platt (see \cite{plattthesis,platt_isolating_2017} or \cite{LMFDB}) based on his rigorous isolation of all the non-trivial zeros of $\zeta$ below about $3\times 10^{10}$. The most extreme value found was $|S(T)|=2.5167$.

Unfortunately, most verifications of RH are content simply to isolate the zeros sufficiently to detect a sign change in $Z(t)$. Isolating the zeros further would increase the run time. Furthermore, the computation reported by Platt and Trudgian \cite{PlaTru21RH} for example, verifying RH to $10^{12}$, considered nearly $4\times 10^{12}$ zeros and even at IEEE double precision, these would require $16$ Tbytes of storage. Thus to go beyond $T=3\times 10^{10}$, it was necessary to isolate all the zeros and compute $S(T)$ at each one. We used the ARB ball arithmetic package \cite{ARB}, specifically the routine acb\_dirichlet\_platt\_hardy\_z\_zeros which is an implementation of Platt's rigorous zero isolation method. The task parallelises trivially to enable us to make use of the multi-core nodes on typical High Performance Clusters and we were fortuante to be given time on both the University Of Bristol's BlueCrystal IV \cite{ACRC} and NCI Australia's Gadi \cite{NCI}.

The maximal value (in absolute terms) of $S(T)$ found was $-2.5682777\ldots$ just before zero number $136\,730\,160\,949$ at $T=39\,828\,558\,846.4263\ldots$. The fact that increasing the search from $3\times 10^{10}$ to $10^{11}$ only moved $S(T)$ by $1.5\%$ is in line with the expected $\log\log$ growth rate of $S(T)$. Indeed, $\log\log (10^{11})$ is about $2\%$ larger than $\log\log (3\times 10^{10})$.

\section{Preliminaries}
In this section, we recall some known results involving the Riemann zeta-function, which will be used throughout the proof of our theorems.

As explained in \cite{Tr14-2, HASANALIZADE2022219}, bounds of $\zeta(\sigma+it)$ when $\sigma=\frac{1}{2}$ and $\sigma=1$ play a crucial role.
The sharpest known estimates for $\sigma=1$ and $t\ge 3$ are the  following:\footnote{Recently Qingyi and Teo \cite{qingyi2024} found a new explicit bound for $|\zeta(1+it)|$, which might lead to some small improvements.}
\begin{equation}\label{boundon1}
    |\zeta(1+i t)| \leq \left\{\begin{array}{lll}
    \min \{\log t, \frac{1}{2} \log t+1.93, \frac{1}{5} \log t+44.02\} & \text{if }3\le t\le \exp(3070);\\ & \\
        1.731 \frac{\log t}{\log \log t} & \text{if }\exp({3070}) < t \leq  \exp(3.69 \cdot 10^8);% \approx 10^{1.6 \cdot 10^8} 
        \\
         & \\
         58.096\log^{2/3}t & \text{if } t>\exp(3.69 \cdot 10^8).
    \end{array}\right.
\end{equation}
In order of appearance, the estimates are due to Patel \cite{patel_explicit_2022}, Hiary--Leong--Yang \cite{hiary2023explicit} and the last one is a direct consequence of  \cite[Theorem 1]{BELLOTTI2024128249}. Nevertheless, to apply Lemma \ref{lindprinc}, we actually require an estimate of the form
\begin{equation}\label{notation1line}
    |\zeta(1+i t)|\le c_1(\log t)^{c_2} 
\end{equation}
for $t\ge t_0\ge e$, where $c_1,c_2$ are constants independent of $t$. Hence, for the middle range $\exp (3070) < t \leq \exp \left(3.69 \cdot 10^8\right)$ we will use the slightly worse upper bound $|\zeta(1+it)|\le 0.25\log t$.

The sharpest known estimates for $|\zeta(\frac{1}{2}+it)|$ up to date are instead:
    \begin{equation*}
\left|\zeta\left(\frac{1}{2}+i t\right)\right| \leq\left\{\begin{array}{ll}
1.461 & \text { if } 0 \leq|t| \leq 3; \\ \\
0.618|t|^{1 / 6} \log |t| & \text { if }3<|t| \le \exp(105);\\ \\
 66.7 t^{27 / 164}  & \text { if } |t| > \exp(105).
\end{array}\right.
\end{equation*}
In order of appearance, these estimates are due to Hiary \cite{hiary_explicit_2016}, Hiary--Patel--Yang \cite{hiary_improved_2024}, Patel--Yang \cite{patel2023explicit}.
We will write
\begin{equation}\label{notation1/2}
    \left|\zeta\left(\frac{1}{2}+i t\right)\right| \leq k_1 t^{k_2}(\log t)^{k_3}
\end{equation}
for $t\ge t_1\ge e$, where $k_1,k_2,k_3$ are constants independent of $t$. Meanwhile, there exists another set of vertical lines inside the critical strip on which explicit bounds for $\zeta(s)$ are known. More precisely, Yang \cite{YANG2024128124} provided the following explicit estimate for $ \left|\zeta\left(\sigma_k+i t\right)\right| $, with $\sigma_k:=1-k /\left(2^k-2\right)$, for every $k\ge 4$: 
\begin{equation}\label{yangintbound}
    \left|\zeta\left(\sigma_k+i t\right)\right| \leq 1.546 t^{1 /\left(2^k-2\right)} \log t, \quad t \geq 3 .
\end{equation}
For instance, substituting $k=4$ gives $$|\zeta(5 / 7+i t)| \leq 1.546 t^{1 / 14} \log t.$$
The bound \eqref{yangintbound} is our main innovative tool, leading to improved values of both $C_1$ and $C_2$ in Theorem \ref{th1}.

\subsection{Other useful results}
We recall other results that will be used to prove Theorem \ref{th1} and Theorem \ref{generalth1}. First of all, for $N \in \mathbb{N}$, we consider the following function that will be fundamental in our argument:
$$
f_N(s)=\frac{1}{2}\left(((s+i T-1) \zeta(s+i T))^N+((s-i T-1) \zeta(s-i T))^N\right).
$$
Then, denoting with $D(c, r)$ the open disk centred at $c$ with radius $r$, for any $N \in \mathbb{N}$, we define
$$
S_N(c, r)=\frac{1}{N} \sum_{z \in \mathcal{S}_N(D(c, r))} \log \frac{r}{|z-c|},
$$
where $\mathcal{S}_N(D(c, r))$ is the set of zeros of $f_N(s)$ in $D(c, r)$. 
\begin{lemma}[\cite{Hasanalizade2021CountingZO} Prop. 3.5]\label{prop3.1}
Let $c, r$, and $\sigma_1$ be real numbers such that
$$
c-r<\frac{1}{2}<1<c<\sigma_1<c+r .
$$
Let $F_{c, r}:[-\pi, \pi] \rightarrow \mathbb{R}$ be an even function such that $F_{c, r}(\theta) \geq \frac{1}{N_m} \log \left|f_{N_m}\left(c+r e^{i \theta}\right)\right|$. Then there is an infinite sequence of natural numbers $\left(N_m\right)_{m=1}^{\infty}$ such that
$$
\limsup _{m \rightarrow \infty} S_{N_m}(c, r) \leq \log \left(\frac{1}{\sqrt{(c-1)^2+T^2}} \frac{\zeta(c)}{\zeta(2 c)}\right)+\frac{1}{\pi} \int_0^\pi F_{c, r}(\theta) d \theta .
$$
\end{lemma}
We remark that Lemma \ref{prop3.1} uses the following trivial lower bound for zeta when $\sigma>1$:
\begin{equation}\label{triviallower}
    |\zeta(\sigma+iT)|\ge \frac{\zeta(2\sigma)}{\zeta(\sigma)}.
\end{equation}
Furthermore, we can alternatively invoke the lower bound of the form
\begin{equation}\label{zeta-lower}
    |\zeta(\sigma+iT)|>\frac{1}{b\log T},
\end{equation}
for $\sigma>1$, which does not depend on the real part $\sigma$ but only on the imaginary part $T$. Indeed, recently Leong \cite{leong2024explicitestimateslogarithmicderivative} showed that \eqref{zeta-lower} is valid with $b=24.302$ for $T>30610046000$. 

Now, instead of using the trivial bound \eqref{triviallower}, injecting Leong's lower bound  \cite{leong2024explicitestimateslogarithmicderivative} into the proof of Lemma \ref{prop3.1}, we derive the following.

\begin{lemma}\label{newprop3.1}
Under the same assumption as in Lemma \ref{prop3.1}, if  $T>30610046000$ there is an infinite sequence of natural numbers $\left(N_m\right)_{m=1}^{\infty}$ such that
$$
\limsup _{m \rightarrow \infty} S_{N_m}(c, r) \leq \log \left(\frac{b\log T}{\sqrt{(c-1)^2+T^2}} \right)+\frac{1}{\pi} \int_0^\pi F_{c, r}(\theta) d \theta 
$$
with $b=24.302$.
\end{lemma}
Backlund's trick is another important tool we will use for the proof of Theorem \ref{th1}. To end this section, we recall the following version of Backlund's trick (see \cite[Prop. 3.7]{Hasanalizade2021CountingZO}).

\begin{lemma}[Backlund's trick]\label{backlund} Let $c$ and $r$ be real numbers. Set
$$
\sigma_1=c+\frac{(c-1 / 2)^2}{r} \quad \text { and } \quad \delta=2 c-\sigma_1-\frac{1}{2}.
$$
If $1<c<r$ and $0<\delta<\frac{1}{2}$, then
$$
\left|\arg ((\sigma+i T-1) \zeta(\sigma+i T))|_{\sigma=\sigma_1}^{1 / 2} \right| \leq \frac{\pi S_N(c, r)}{2 \log (r /(c-1 / 2))}+\frac{E(T, \delta)}{2}+\frac{\pi}{N}+\frac{\pi}{2 N}+\frac{\pi}{4},
$$
where for $0\le d < 9/2$ and $T\ge 5/7$, $E(T,d)$ is defined as in \cite[p. 1463]{bennett_counting_2021}:
\begin{align*}
 \begin{split}
E (T,d)& = \frac{2T /3}{(2d + 17)^2 + 4 T^2} + \frac{2T/3}{( -2d +17)^2 + 4T^2}  - \frac{4T/3}{17^2 + 4T^2}\\
& + \frac{T}{2} \log \left( 1 + \frac{ 17^2}{4T^2}\right) - \frac{T}{4} \log \left( 1 + \frac{( 2d +17)^2}{4T^2}\right) - \frac{T}{4} \log \left( 1 + \frac{( -2d + 17)^2}{4T^2}\right)\\
&+ \frac{(8+6\pi)/45}{(( 2d + 17)^2 + 4T^2)^{3/2}}
 +  \frac{(8+6\pi)/45}{(( - 2d + 17)^2 + 4T^2)^{3/2}} + \frac{2(8+ 6\pi)/45}{( 17^2 + 4T^2)^{3/2}} \\
& + \sum_{k=0}^3\left( 2\arctan \frac{1+4k}{2T} - \arctan \frac{2d+1+4k}{2T} - \arctan \frac{-2d+1+4k}{2T} \right)\\  
& + \frac{ 2d +15}{4} \arctan \frac{ 2d + 17}{2T}  + \frac{ - 2d +15}{4} \arctan \frac{ -2d + 17}{2T}
 - \frac{ 15}{2} \arctan\frac{17}{2T}.
 \end{split} 
\end{align*}
\end{lemma}
We also recall the following estimate for $E(T,\delta)/\pi$ due to \cite[Lemma 3.4]{bennett_counting_2021}.
\begin{lemma}[\cite{bennett_counting_2021}]
For $0 \leq \delta_1 \leq d<9 / 2$ and $T \geq 5 / 7$, one has
$$
0<E\left(T, \delta_1\right) \leq E(T, d)
$$
Also, for $d \in\left[\frac{1}{4}, \frac{5}{8}\right]$ and $T \geq 5 / 7$, one has
$$
\frac{E(T, d)}{\pi} \leq \frac{640 d-112}{1536(3 T-1)}+\frac{1}{2^{10}}.
$$
\end{lemma}
To conclude the section, we mention the Phragm\'en--Lindel\"of principle, as stated in \cite[Lemma 3]{Tr14-2},\footnote{There is a typo in the statement of Proposition 4.1 of \cite{HASANALIZADE2022219}, where $\log|Q+s|$ should be replace with $|\log(Q+s)|$.} which will be used to construct a suitable function $F_{c,r}(\theta)$, given Lemmata \ref{prop3.1}, \ref{newprop3.1}, and \ref{backlund}.

\begin{lemma}[Phragm\'en-Lindel\"of principle]\label{lindprinc}
Let $a, b, Q$ be real numbers such that $b>$ $a$ and $Q+a>1$. Let $f(s)$ be a holomorphic function on the strip $a \leq \mathfrak{R e}(s) \leq b$ such that
$$
|f(s)|<C \exp \left(e^{k|t|}\right)
$$
for some $C>0$ and $0<k<\frac{\pi}{b-a}$. Suppose, further, that there are $A, B, \alpha_1, \alpha_2, \beta_1, \beta_2 \geq 0$ such that $\alpha_1 \geq \beta_1$ and
$$
|f(s)| \leq\left\{\begin{array}{ll}
A|Q+s|^{\alpha_1}(\log |Q+s|)^{\alpha_2} & \text { for } \mathfrak{R e}(s)=a ; \\
B|Q+s|^{\beta_1}(\log |Q+s|)^{\beta_2} & \text { for } \mathfrak{R e}(s)=b .
\end{array}\right.
$$
Then for $a \leq \mathfrak{R e}(s) \leq b$, one has
$$
|f(s)| \leq\left\{A|Q+s|^{\alpha_1}|\log (Q+s)|^{\alpha_2}\right\}^{\frac{b-\mathfrak{R e}(s)}{b-a}}\left\{B|Q+s|^{\beta_1}|\log (Q+s)|^{\beta_2}\right\}^{\frac{\mathfrak{R e}(s)-a}{b-a}} .
$$
\end{lemma}
For our purpose, it will be more convenient to use the following inequality, which holds for every $\mathfrak{Im}(Q+s)\ge 30610046000$:
\begin{equation}\label{eq:plineq}
    |\log (Q+s)|\le 1.00212 \log |Q+s|.
\end{equation}

\section{Proof of Theorem \ref{generalth1}}
As in \cite{Tr14-2,HASANALIZADE2022219}, we will work with the completed Riemann zeta-function $\xi(s)$, defined by
$$
\xi(s)=s(s-1) \gamma(s) \zeta(s)\quad\text{with}\quad 
\gamma(s)=\pi^{-\frac{s}{2}} \Gamma\left(\frac{s}{2}\right).
$$
It is widely known that the function $\xi(s)$ is an entire function of order 1  and satisfies the functional equation
$
\xi(s)=\xi(1-s).
$
 Furthermore, as in \cite{HASANALIZADE2022219}, instead of working directly with the quantity $N(T)$, we will work with the quantity $N_{\mathbb{Q}}(T)$ defined as
$$
N_{\mathbb{Q}}(T)=\#\{\rho \in \mathbb{C}\mid\zeta(\rho)=0,0<\beta<1,| \gamma \mid \leq T\}
$$
for $T \geq 0$. It is related to the quantity $N(T)$ via the equality $N_{\mathbb{Q}}(T)=2 N(T)$. Following \cite{HASANALIZADE2022219}, given a parameter $\sigma_1>1$ we will choose later, we consider the rectangle $\mathcal{R}$ with vertices $\sigma_1-i T$, $\sigma_1+i T$, $1-\sigma_1+i T$, and $1-\sigma_1-i T$. As $\xi(s)$ is entire, the argument principle yields
$$
N_{\mathbb{Q}}(T)=\frac{1}{2 \pi} \Delta_{\mathcal{R}} \arg \xi(s).
$$
As per \cite[Eq. (2.8)]{HASANALIZADE2022219}, the previous expression for $N_{\mathbb{Q}}(T)$ is equivalent to
\begin{equation*}
N_{\mathbb{Q}}(T)=\frac{2}{\pi} \arctan (2 T)+g(T)+\frac{T}{\pi} \log \left(\frac{T}{2 \pi e}\right)-\frac{1}{4}+\frac{2}{\pi} \Delta_{\mathcal{C}_0} \arg ((s-1) \zeta(s)),
\end{equation*}
where 
\begin{equation*}
g(T)=\frac{2}{\pi} \mathfrak{Im} \log \Gamma\left(\frac{1}{4}+i \frac{T}{2}\right)-\frac{T}{\pi} \log \left(\frac{T}{2 e}\right)+\frac{1}{4},
\end{equation*}
and $\mathcal{C}_0$ is the part of the contour of $\mathcal{R}$ in  the region $\mathfrak{Im}(s)\ge 0$ and $\mathfrak{Re}(s)\ge \frac{1}{2}$. The estimate $|g(T)|\le 1/(25T)$ is given by \cite[Prop. 3.2]{bennett_counting_2021} and holds for every $T\ge 5/7$. Then, denoting with $\mathcal{C}_1$ the vertical line from $\sigma_1$ to $\sigma_1+iT$ and with $\mathcal{C}_2$ the horizontal line from $\sigma_1+iT$ to $\frac{1}{2}+iT$, one has
\begin{align*}
    \Delta_{\mathcal{C}_0} \arg ((s-1) \zeta(s))&=\Delta_{\mathcal{C}_1} \arg ((s-1) \zeta(s))+\Delta_{\mathcal{C}_2} \arg ((s-1) \zeta(s))\\&=\arctan \left(\frac{T}{\sigma_1-1}\right)+\Delta_{\mathcal{C}_1} \arg \zeta(s)+\Delta_{\mathcal{C}_2} \arg ((s-1) \zeta(s))
\end{align*}
and for $\sigma_1>1$,
$$
\left|\Delta_{\mathcal{C}_1} \arg \zeta(s)\right|=\left|\arg \zeta\left(\sigma_1+i T\right)\right| \leq\left|\log \zeta\left(\sigma_1+i T\right)\right| \leq \log \zeta\left(\sigma_1\right).
$$
Furthermore, we recall \cite[Eq. (5.5)]{HASANALIZADE2022219} asserting
\begin{equation*}
S(T)=\frac{1}{\pi} \Delta_{\mathcal{C}_0} \arg \zeta(s)=\frac{1}{2}\left(N_{\mathbb{Q}}(T)-\frac{T}{\pi} \log \left(\frac{T}{2 \pi e}\right)+\frac{1}{4}-g(T)-2\right)
\end{equation*}
and
\begin{equation*}\label{eqS(T)}
   S(T)=\frac{1}{\pi} \Delta_{\mathcal{C}_0} \arg \zeta(s)=\frac{1}{\pi} \Delta_{\mathcal{C}_1} \arg \zeta(s)+\frac{1}{\pi} \Delta_{\mathcal{C}_2} \arg (s-1) \zeta(s)-\frac{1}{\pi} \Delta_{\mathcal{C}_2} \arg (s-1). 
\end{equation*}

Now, instead of estimating the usual quantity 
\[
\left|N_{\mathbb{Q}}(T)-\frac{T}{\pi} \log \left(\frac{T}{2 \pi e}\right)+\frac{1}{4}\right|
\]
as in \cite{HASANALIZADE2022219} (or $\frac{1}{8}$ instead of $\frac{1}{4}$ if we work with $N(T)$ instead of $N_{\mathbb{Q}}(T)$ as in \cite{Tr14-2}), we aim to estimate 
\[
\left|N_{\mathbb{Q}}(T)-\frac{T}{\pi} \log \left(\frac{T}{2 \pi e}\right)\right|,
\]
which is the quantity that is usually required in the majority of applications.
\begin{remark}
    This choice will allow us to save an extra $0.25$ in the final constant $C_3$ in Theorem \ref{th1}.
\end{remark}
One has
\begin{equation}\label{beforeintegral}
    \begin{aligned}
        &\left|N_{\mathbb{Q}}(T)-\frac{T}{\pi} \log \left(\frac{T}{2 \pi e}\right)\right|
        \\&\le \left|\frac{1}{4}-\frac{2}{\pi}\arctan \left(\frac{T}{\sigma_1-1}\right)-\frac{2}{\pi} \arctan (2 T)\right|+ |g(T)|+\frac{2}{\pi} \log \zeta\left(\sigma_1\right)+\frac{2}{\pi}\left|\Delta_{\mathcal{C}_2} \arg ((s-1) \zeta(s))\right|
        \\&\le \left|\frac{1}{4}-\frac{2}{\pi}\arctan \left(\frac{T}{\sigma_1-1}\right)-\frac{2}{\pi} \arctan (2 T)\right|+ \frac{1}{25T}+\frac{2}{\pi} \log \zeta\left(\sigma_1\right)+\frac{2}{\pi}\left|\Delta_{\mathcal{C}_2} \arg ((s-1) \zeta(s))\right|\\&=|h(c,r,T)|+ \frac{1}{25T}+\frac{2}{\pi} \log \zeta\left(\sigma_1\right)+\frac{2}{\pi}\left|\Delta_{\mathcal{C}_2} \arg ((s-1) \zeta(s))\right|,
    \end{aligned}
\end{equation}
where 
$$
h(c,r,T)= \frac{1}{4}-\frac{2}{\pi}\arctan \left(\frac{T}{\sigma_1-1}\right)-\frac{2}{\pi}\arctan (2 T).
$$
Since $\sigma_1>1$ (it will be approximately $\sigma_1\approx 1.25$), and $T$ will be quite large (namely, $T>30610046000$), by the fact that the function $\arctan$ increases for $T>0$, we can bound $|h(c,r,T)|$ as
\begin{equation}\label{boundhcrt}
    |h(c,r,T)|\le \left|\frac{2}{\pi}\arctan (4 T)+\frac{2}{\pi}\arctan (2 T)-\frac{1}{4}\right|\le 2-\frac{1}{4}=\frac{7}{4}.
\end{equation}

In the next section, we will focus on improving the bound for
\[
\left|\Delta_{\mathcal{C}_2} \arg ((s-1) \zeta(s))\right|.
\]

\subsection{Convexity/subconvexity bounds and generalisation to $n$-splitting of $[\frac{1}{2},1]$}\label{nsplitting}
While outside the range $[\frac{1}{2},1]$ we consider the same cases as in \cite{HASANALIZADE2022219}, we will split the interval $[\frac{1}{2},1]$ in $n$ sub-intervals using the lines corresponding to \eqref{yangintbound}. As already mentioned, this splitting will be the main tool which allows us to get an improved value of $C_1$.  More precisely, the first sub-interval is of the form $[\frac{1}{2},\sigma_4]$, the last interval will be of the form $[\sigma_{n+4},1]$ and the intervals in the middle will be of the form $[\sigma_{4+h},\sigma_{4+h+1}]$, where $0\le h\le n-1 $ and, for every $k$, $\sigma_k$ is defined as
$$\sigma_k:=1-k /(2^k-2).$$
A careful analysis reveals that the nearly optimal value for $n$ is 5. Indeed, a higher number of subintervals of the form $[\sigma_{4+h},\sigma_{4+h+1}]$ would lead to an improvement on $C_1$ only at the seventh decimal place, while it would cause a worse constant $C_3$, in which a factor of containing $\log(1.546)$ (multiplied by other quantities) appears in $C_3$ for each interval $[\sigma_{4+h},\sigma_{4+h+1}]$ we are considering.\\
Following \cite{HASANALIZADE2022219}, we start estimating $|\zeta(\sigma+it)|$ in each of the intervals for $\sigma$ we are considering.\footnote{\cite{fiori2025notephragmenlindeloftheorem} addresses some possible errors in the literature related to the Phragm\'en-Lindel\"of principle. However, all the bounds for $|\zeta(\sigma+it)|$ obtained in the current paper by an application of Lemma \ref{lindprinc} can be recovered by following the method outlined in \cite{fiori2025notephragmenlindeloftheorem}.}
\begin{itemize}
    \item Case $\sigma \geq 1+\eta$. The trivial bound for the zeta-function immediately implies
$
|\zeta(s)| \leq \zeta(\sigma) .
$

\item Case $1 \leq \sigma \leq 1+\eta$. 
From \eqref{boundon1} and \eqref{notation1line}
 it follows that there is $Q_0>0$ depending on $c_1, c_2, t_0$ such that
 \begin{equation}\label{Q0}
     |(1+i t-1) \zeta(1+i t)| \leq c_1\left|Q_0+(1+i t)\right|\left(\log \left|Q_0+(1+i t)\right|\right)^{c_2}
 \end{equation}
for all $t$. Thus, Lemma \ref{lindprinc} and \eqref{eq:plineq} imply that for $1 \leq \sigma \leq 1+\eta$,
$$
|(s-1) \zeta(s)| \leq\left(c_1\left|Q_0+s\right|\left(1.00212\log \left|Q_0+s\right|\right)^{c_2}\right)^{\frac{1+\eta-\sigma}{\eta}}\left(\zeta(1+\eta)\left|Q_0+s\right|\right)^{\frac{\sigma-1}{\eta}},
$$
and hence
$$
|\zeta(s)| \leq \frac{1}{|s-1|}\left(c_1\left|Q_0+s\right|\left(1.00212\log \left|Q_0+s\right|\right)^{c_2}\right)^{\frac{1+\eta-\sigma}{\eta}}\left(\zeta(1+\eta)\left|Q_0+s\right|\right)^{\frac{\sigma-1}{\eta}} .
$$

\item Case $\sigma_{n+4}\le \sigma\le 1$. 
By \eqref{yangintbound}, there exists $Q_{n+4}>0$ depending on $n$ such that
\begin{equation}\label{Qnp4}
\begin{aligned}
       &|(\sigma_{n+4}+it-1)\zeta(\sigma_{n+4}+it)|\\&\le  1.546|Q_{n+4}+(\sigma_{n+4}+it)|^{\frac{1}{2^{n+4}-2}+1}(\log |Q_{n+4}+(\sigma_{n+4}+it)|).
\end{aligned}
\end{equation}
Thus, Lemma \ref{lindprinc}, \eqref{eq:plineq} and \eqref{Q0} imply that for $\sigma_{n+4}\le \sigma\le 1$,
\begin{align*}
    |\zeta(s)|&\le \frac{1}{|s-1|}\left( 1.546|Q_{0,n+4}+s|^{\frac{1}{2^{n+4}-2}+1}(1.00212\log |Q_{0,n+4}+s|)\right)^{\frac{1-\sigma}{1-\sigma_{n+4}}}\\
&\times\left(c_1\left|Q_{0,n+4}+s\right|\left(1.00212\log \left|Q_{0,n+4}+s\right|\right)^{c_2}\right)^{\frac{\sigma-\sigma_{n+4}}{1-\sigma_{n+4}}},
\end{align*}
where $Q_{0,n+4}=\max\{Q_0,Q_{n+4}\}$.

\item Case  $\sigma_{4+h}\le \sigma\le \sigma_{4+h+1}$ with $0\le h\le n-1$. 
By \eqref{yangintbound}, for every fixed $h$, there exist $Q_{4+h},Q_{5+h}>0$, depending on $h$, so that
\begin{equation}\label{Q4ph}
\begin{aligned}
       &|(\sigma_{h+4}+it-1)\zeta(\sigma_{h+4}+it)|\\&\le  1.546|Q_{h+4}+(\sigma_{h+4}+it)|^{\frac{1}{2^{h+4}-2}+1}(\log |Q_{h+4}+(\sigma_{h+4}+it)|)
\end{aligned}
\end{equation}
and
\begin{equation*}
\begin{aligned}
        &|(\sigma_{h+5}+it-1)\zeta(\sigma_{h+5}+it)|\\&\le  1.546|Q_{h+5}+(\sigma_{h+5}+it)|^{\frac{1}{2^{h+5}-2}+1}(\log |Q_{h+5}+(\sigma_{h+5}+it)|).
\end{aligned}
\end{equation*}
Hence, by Lemma \ref{lindprinc} and \eqref{eq:plineq}, we have
\begin{align*}
   & |\zeta(s)|\le \frac{1}{|s-1|}\left( 1.546|Q_{4+h,5+h}+s|^{\frac{1}{2^{h+4}-2}+1}(1.00212\log |Q_{4+h,5+h}+s|)\right)^{\frac{\sigma_{5+h}-\sigma}{\sigma_{5+h}-\sigma_{h+4}}}\\&\qquad\times\left(1.546|Q_{4+h,5+h}+s|^{\frac{1}{2^{h+5}-2}+1}(1.00212\log |Q_{4+h,5+h}+s|)\right)^{\frac{\sigma-\sigma_{h+4}}{\sigma_{h+5}-\sigma_{h+4}}},
\end{align*}
where $Q_{4+h,5+h}=\max\{Q_{4+h},Q_{5+h}\}$.
\begin{remark}
    When $h=n-1$, $Q_{h+5}=Q_{n+4}$, with $Q_{n+4}$ given in \eqref{Qnp4}.
\end{remark}

\item Case $1/2\le \sigma\le \sigma_4$. 
By \eqref{notation1/2}, there is a $Q_1>0$ depending on $k_1,k_2,k_3,t_1$ such that 
\begin{equation*}\label{Q1}
\begin{aligned}
    \left|\left(\frac{1}{2}+i t-1\right) \zeta\left(\frac{1}{2}+i t\right)\right| \leq k_1\left|Q_1+\left(\frac{1}{2}+i t\right)\right|^{k_2+1}\left(\log \left|Q_1+\left(\frac{1}{2}+i t\right)\right|\right)^{k_3}.
\end{aligned}
\end{equation*}
Hence, Lemma \ref{lindprinc}, \eqref{eq:plineq} and \eqref{Q4ph} with $h=0$ imply that
\begin{align*}
   & |\zeta(s)|\\
   &\le \frac{1}{|s-1|}\left(k_1\left|Q_2+s\right|^{k_2+1}\left(1.00212\log \left|Q_2+s\right|\right)^{k_3}\right)^{\frac{\sigma_{4}-\sigma}{\sigma_{4}-\frac{1}{2}}}
    \left( 1.546|Q_2+s|^{\frac{1}{2^4-2}+1}(1.00212\log |Q_2+s|)\right)^{\frac{\sigma-\frac{1}{2}}{\sigma_{4}-\frac{1}{2}}},
\end{align*}
where $Q_2=\max\{Q_1,Q_{4+0}\}$.

\item Case $0\le \sigma\le \frac{1}{2}$. 
Following \cite{HASANALIZADE2022219}, there exists $Q_3>0$ such that
\begin{equation*}
\left|\zeta\left(\frac{1}{2}+i t\right)\right| \leq k_1\left|Q_3+\left(\frac{1}{2}+i t\right)\right|^{k_2}\left(\log \left|Q_3+\left(\frac{1}{2}+i t\right)\right|\right)^{k_3}
\end{equation*}
for all $t$ and $Q_{10}\ge 1$ such that
\begin{equation*}
|\zeta(0+i t)| \leq \frac{c_1}{\sqrt{2 \pi}}\left|Q_{10}+i t\right|^{\frac{1}{2}}\left(\log \left|Q_{10}+i t\right|\right)^{c_2}.
\end{equation*}
Hence, by Lemma \ref{lindprinc} and \eqref{eq:plineq} one has
\begin{equation*}
|\zeta(s)| \leq\left(\frac{c_1}{\sqrt{2 \pi}}\left|Q_{11}+s\right|^{\frac{1}{2}}\left(1.00212\log \left|Q_{11}+s\right|\right)^{c_2}\right)^{1-2 \sigma}\left(k_1\left|Q_{11}+s\right|^{k_2}\left(1.00212\log \left|Q_{11}+s\right|\right)^{k_3}\right)^{2 \sigma}
\end{equation*}
where $Q_{11}=\max\{Q_{3},Q_{10}\}$.
\item Case $-\eta \leq \sigma \leq 0$. 
As in \cite{HASANALIZADE2022219} \footnote{Note that we have an extra factor $1.00212$ due to the typo in Proposition 4.1 of \cite{HASANALIZADE2022219} that we already mentioned before.}, we have
\begin{equation*}
|\zeta(s)| \leq\left(\frac{1}{(2 \pi)^{\frac{1}{2}+\eta}} \zeta(1+\eta)\left|Q_{10}+s\right|^{\frac{1}{2}+\eta}\right)^{\frac{-\sigma}{\eta}}\left(\frac{c_1}{\sqrt{2 \pi}}\left|Q_{10}+s\right|^{\frac{1}{2}}\left(1.00212\log \left|Q_{10}+s\right|\right)^{c_2}\right)^{\frac{\sigma+\eta}{\eta}} .
\end{equation*}

\item Case $\sigma \leq-\eta$. We shall use the same estimate as in \cite{HASANALIZADE2022219}:
\begin{equation*}
|\zeta(s)| \leq \zeta(1-\sigma)\left(\frac{1}{2 \pi}\right)^{\frac{1}{2}-\sigma}(|1+s-[\sigma]|)^{\frac{1}{2}+[\sigma]-\sigma}\prod_{j=1}^{-[\sigma]}|s+j-1|.
\end{equation*}
\end{itemize}
\subsection{Estimating $\frac{1}{N}\log|f_N(s)|$}
Given the function
\begin{equation*}
f_N(s)=\frac{1}{2}\left(((s+i T-1) \zeta(s+i T))^N+((s-i T-1) \zeta(s-i T))^N\right),
\end{equation*}
we aim to bound 
\begin{equation*}
\frac{1}{N} \log \left|f_N(s)\right|
\end{equation*} 
inside the different ranges we considered in the previous subsection.
\begin{itemize}
    \item Case $\sigma\ge 1+\eta$.     As per \cite{HASANALIZADE2022219}, we have the bound
    \begin{equation*}
\frac{1}{N} \log \left|f_N(s)\right| \leq \frac{1}{2} \log \left((\sigma-1)^2+(|t|+T)^2\right)+\log \zeta(\sigma).
\end{equation*}

\item Case $1\le\sigma\le 1+\eta$. Following \cite{HASANALIZADE2022219}, we use
\begin{equation*}
\begin{aligned}
\frac{1}{N} \log \left|f_N(s)\right| & \leq \frac{1+\eta-\sigma}{\eta} \log \left(\frac{c_1 \cdot 1.00212^{c_2}}{2^{c_2}}\right)+\frac{\sigma-1}{\eta} \log \zeta(1+\eta)+\frac{1}{2} \log \left(\left(Q_0+\sigma\right)^2+(|t|+T)^2\right) \\
& +\frac{c_2(1+\eta-\sigma)}{\eta} \log \log \left(\left(Q_0+\sigma\right)^2+(|t|+T)^2\right).
\end{aligned}
\end{equation*}

\item Case $\sigma_{n+4}\le \sigma\le 1$. Observe that
\begin{align*}
    |f_N(s)|&\le \left(1.546\cdot 1.00212((Q_{0,n+4}+\sigma)^2+(|t|+T)^2)^{\frac{2^{n+4}-1}{2(2^{n+4}-2)}}\log(\sqrt{(Q_{0,n+4}+\sigma)^2+(|t|+T)^2})\right)^{\frac{N(1-\sigma)}{1-\sigma_{n+4}}}\\
    &\times \left(c_1((Q_{0,n+4}+\sigma)^2+(|t|+T)^2)^{\frac{1}{2}}\left(1.00212\log(\sqrt{(Q_{0,n+4}+\sigma)^2+(|t|+T)^2})\right)^{c_2}\right)^{\frac{N(\sigma-\sigma_{n+4})}{1-\sigma_{n+4}}}.
\end{align*}
Hence, taking the logarithms of both sides and dividing by $N$, we obtain
\begin{align*}
    &\frac{1}{N} \log \left|f_N(s)\right|\le \frac{(1-\sigma)}{1-\sigma_{n+4}}\log\left(1.546\cdot 1.00212\right)+\frac{(\sigma-\sigma_{n+4})}{1-\sigma_{n+4}}\log\left(c_1\cdot 1.00212^{c_2}\right)\\
    &-\left(\frac{(1-\sigma)}{(1-\sigma_{n+4})}+\frac{c_2(\sigma-\sigma_{n+4})}{(1-\sigma_{n+4})}\right)\log 2\\&+\left(\frac{(2^{n+4}-1)(1-\sigma)}{2(2^{n+4}-2)(1-\sigma_{n+4})}+\frac{(\sigma-\sigma_{n+4})}{2(1-\sigma_{n+4})}\right)\log \left(\left(Q_{0,n+4}+\sigma\right)^2+(|t|+T)^2\right)\\&+\left(\frac{(1-\sigma)}{(1-\sigma_{n+4})}+\frac{c_2(\sigma-\sigma_{n+4})}{(1-\sigma_{n+4})}\right)\log\log \left(\left(Q_{0,n+4}+\sigma\right)^2+(|t|+T)^2\right).
\end{align*}
\item Case $\sigma_{4+h}\le \sigma\le \sigma_{4+h+1}$, where $0\le h\le n-1$. 
It follows from 
\begin{align*}
    &|f_N(s)|\\&\le \left(1.546\cdot 1.00212((Q_{4+h,5+h}+\sigma)^2+(|t|+T)^2)^{\frac{2^{h+4}-1}{2(2^{h+4}-2)}}\log(\sqrt{(Q_{4+h,5+h}+\sigma)^2+(|t|+T)^2})\right)^{\frac{N(\sigma_{5+h}-\sigma)}{\sigma_{5+h}-\sigma_{h+4}}}\\
    &\times (1.546\cdot 1.00212((Q_{4+h,5+h}+\sigma)^2+(|t|+T)^2)^{\frac{2^{h+5}-1}{2(2^{h+5}-2)}}(\log(\sqrt{(Q_{4+h,5+h}+\sigma)^2+(|t|+T)^2})))^{\frac{N(\sigma-\sigma_{h+4})}{\sigma_{5+h}-\sigma_{h+4}}}
\end{align*}
that
\begin{align*}
    &\frac{1}{N} \log \left|f_N(s)\right|\\&\le \left(\frac{(\sigma_{5+h}-\sigma)}{\sigma_{5+h}-\sigma_{h+4}}+\frac{(\sigma-\sigma_{h+4})}{\sigma_{5+h}-\sigma_{h+4}}\right)(\log(1.546\cdot 1.00212))-\left(\frac{(\sigma_{5+h}-\sigma)}{\sigma_{5+h}-\sigma_{h+4}}+\frac{(\sigma-\sigma_{h+4})}{\sigma_{5+h}-\sigma_{h+4}}\right)\log 2\\&+\left(\frac{(2^{h+4}-1)(\sigma_{5+h}-\sigma)}{2(2^{h+4}-2)(\sigma_{5+h}-\sigma_{h+4})}+\frac{(2^{h+5}-1)(\sigma-\sigma_{h+4})}{2(2^{h+5}-2)(\sigma_{5+h}-\sigma_{h+4})}\right)\log \left(\left(Q_{4+h,5+h}+\sigma\right)^2+(|t|+T)^2\right)\\&+\left(\frac{(\sigma_{5+h}-\sigma)}{\sigma_{5+h}-\sigma_{h+4}}+\frac{(\sigma-\sigma_{h+4})}{\sigma_{5+h}-\sigma_{h+4}}\right)\log\log \left(\left(Q_{4+h,5+h}+\sigma\right)^2+(|t|+T)^2\right)\\&=\log (1.546\cdot 1.00212)-\log 2+\log\log \left(\left(Q_{4+h,5+h}+\sigma\right)^2+(|t|+T)^2\right)\\&+\left(\frac{(2^{h+4}-1)(\sigma_{5+h}-\sigma)}{2(2^{h+4}-2)(\sigma_{5+h}-\sigma_{h+4})}+\frac{(2^{h+5}-1)(\sigma-\sigma_{h+4})}{2(2^{h+5}-2)(\sigma_{5+h}-\sigma_{h+4})}\right)\log \left(\left(Q_{4+h,5+h}+\sigma\right)^2+(|t|+T)^2\right).
\end{align*}

\item Case $1/2\le \sigma\le \sigma_{4}$. One has
\begin{align*}
    |f_N(s)|&\le \left(k_1((Q_2+\sigma)^2+(|t|+T)^2)^{\frac{k_2+1}{2}}\left(1.00212\log(\sqrt{(Q_2+\sigma)^2+(|t|+T)^2})\right)^{k_3}\right)^{\frac{N(\sigma_{4}-\sigma)}{\sigma_{4}-\frac{1}{2}}}\\&\times \left(1.546\cdot 1.00212((Q_2+\sigma)^2+(|t|+T)^2)^{\frac{15}{28}}(\log(\sqrt{(Q_2+\sigma)^2+(|t|+T)^2}))\right)^{\frac{N(\sigma-\frac{1}{2})}{\sigma_{4}-\frac{1}{2}}}
\end{align*}
and thus
\begin{align*}
    &\frac{1}{N} \log \left|f_N(s)\right|\\
    &\le \frac{(\sigma_{4}-\sigma)}{\sigma_{4}-\frac{1}{2}}(\log (k_1\cdot 1.00212^{k_3}))+\frac{(\sigma-\frac{1}{2})}{\sigma_{4}-\frac{1}{2}}(\log(1.546\cdot 1.00212)) -\left(\frac{k_3(\sigma_{4}-\sigma)}{\sigma_{4}-\frac{1}{2}}+\frac{(\sigma-\frac{1}{2})}{(\sigma_{4}-\frac{1}{2})}\right)\log 2\\&+\left(\frac{(k_2+1)(\sigma_{4}-\sigma)}{2(\sigma_{4}-\frac{1}{2})}+\frac{15(\sigma-\frac{1}{2})}{28(\sigma_{4}-\frac{1}{2})}\right)\log \left(\left(Q_2+\sigma\right)^2+(|t|+T)^2\right)\\&+\left(\frac{k_3(\sigma_{4}-\sigma)}{\sigma_{4}-\frac{1}{2}}+\frac{(\sigma-\frac{1}{2})}{(\sigma_{4}-\frac{1}{2})}\right)\log\log \left(\left(Q_2+\sigma\right)^2+(|t|+T)^2\right).
\end{align*}

\item Case $0\le\sigma\le1/2$. As in \cite{HASANALIZADE2022219}, we have
\begin{equation*}
\begin{aligned}
&\frac{1}{N} \log \left|f_N(s)\right| \\
&\leq(1-2 \sigma) \log \left(\frac{c_1\cdot 1.00212^{c_2}}{2^{c_2+\frac{1}{2}} \sqrt{\pi}}\right)+2 \sigma \log \left(\frac{k_1\cdot 1.00212^{k_3}}{2^{k_3}}\right)+\frac{1}{2} \log \left((\sigma-1)^2+(|t|+T)^2\right) \\
& +\frac{1-2 \sigma+4 k_2 \sigma}{4} \log \left(\left(Q_{11}+\sigma\right)^2+(|t|+T)^2\right) \\
& +\left(c_2(1-2 \sigma)+2 k_3 \sigma\right) \log \log \left(\left(Q_{11}+\sigma\right)^2+(|t|+T)^2\right)
\end{aligned}
\end{equation*}

\item Case $-\eta\le\sigma\le 0$. By \cite{HASANALIZADE2022219}, we know
\begin{equation*}
\begin{aligned}
\frac{1}{N} \log \left|f_N(s)\right| & \leq-\frac{\sigma}{\eta} \log \left(\frac{1}{(2 \pi)^{\frac{1}{2}+\eta}}\right)-\frac{\sigma}{\eta} \log (1+\eta)+\frac{\sigma+\eta}{\eta} \log\left( \frac{c_1\cdot 1.00212^{c_2}}{\sqrt{2 \pi}}\right)-\frac{\sigma+\eta}{\eta} c_2 \log 2 \\
& +\frac{1}{2} \log \left((\sigma-1)^2+(|t|+T)^2\right)   +\left(-\frac{\sigma(1+2 \eta)}{4 \eta}+\frac{\sigma+\eta}{4 \eta}\right) \log \left(\left(Q_{10}+\sigma\right)^2+(|t|+T)^2\right) \\
& +\frac{\sigma+\eta}{\eta} c_2 \log \log \left(\left(Q_{10}+\sigma\right)^2+(|t|+T)^2\right).
\end{aligned}
\end{equation*}

\item Case $\sigma\le -\eta$. Finally, as per \cite{HASANALIZADE2022219}, we use the bound
\begin{equation*}
\begin{aligned}
\frac{1}{N} \log \left|f_N(s)\right| & \leq \log \zeta(1-\sigma)+\frac{1}{2} \log \left((\sigma-1)^2+(|t|+T)^2\right) \\
& +\frac{2 \sigma-1}{2} \log 2 \pi+\frac{(1-2 \sigma+2[\sigma])}{4} \log \left((1+\sigma-[\sigma])^2+(|t|+T)^2\right) \\
& +\frac{1}{2} \sum_{j=1}^{-[\sigma]} \log \left((\sigma+j-1)^2+(|t|+T)^2\right).
\end{aligned}
\end{equation*}
\end{itemize}

\subsection{Defining $F_{c,r}(\theta)$}
We start recalling some auxiliary functions already defined in \cite{HASANALIZADE2022219} which will appear in the definition of $F_{c,r}(\theta)$. For $\theta \in$ $[-\pi, \pi]$, we let $\sigma=c+r \cos \theta$, with $c-r>-\frac{1}{2}$, and $t=r \sin \theta$. We define
\begin{equation}\label{Lj}
    L_j(\theta)=\log \frac{(j+c+r \cos \theta)^2+(|r \sin \theta|+T)^2}{T^2}
\end{equation}
and
\begin{equation}\label{Mj}
    M_j(\theta)=\log \log \left((j+c+r \cos \theta)^2+(|r \sin \theta|+T)^2\right)-\log \log \left(T^2\right) .
\end{equation}
\begin{itemize}
    \item If $\sigma\ge 1+\eta$, as in \cite{HASANALIZADE2022219}, we define
    $$
F_{c, r}(\theta)=\frac{1}{2} L_{-1}(\theta)+\log T+\log \zeta(\sigma).
$$
\item For $1\le \sigma\le 1+\eta$, as per \cite{HASANALIZADE2022219}
\begin{equation*}
\begin{aligned}
F_{c, r}(\theta) & =\frac{1+\eta-\sigma}{\eta} \log (c_1\cdot 1.00212^{c_2})+\frac{\sigma-1}{\eta} \log \zeta(1+\eta)+\frac{1}{2} L_{Q_0}(\theta)+\log T \\
& +\frac{c_2(1+\eta-\sigma)}{\eta} M_{Q_0}(\theta)+\frac{c_2(1+\eta-\sigma)}{\eta} \log \log T .
\end{aligned}
\end{equation*}
\item If $\sigma_{n+4}\le \sigma\le 1$, then we define
\begin{align*}
    F_{c,r}(\theta)&= \frac{(1-\sigma)}{1-\sigma_{n+4}}\log(1.546\cdot 1.00212)+\frac{(\sigma-\sigma_{n+4})}{1-\sigma_{n+4}}\log (c_1 \cdot 1.00212^{c_2})\\&+\left(\frac{(2^{n+4}-1)(1-\sigma)}{(2^{n+4}-2)(1-\sigma_{n+4})}+\frac{(\sigma-\sigma_{n+4})}{(1-\sigma_{n+4})}\right)\left(\frac{L_{Q_{0,n+4}}(\theta)}{2}+\log T\right)\\&+\left(\frac{(1-\sigma)}{(1-\sigma_{n+4})}+\frac{c_2(\sigma-\sigma_{n+4})}{(1-\sigma_{n+4})}\right)(M_{Q_{0,n+4}}(\theta)+\log\log T).
\end{align*}
\item When $\sigma_{4+h}\le \sigma\le \sigma_{4+h+1}$, where $0\le h\le n-1$, we have
\begin{align*}
    F_{c,r}(\theta)&=\log (1.546\cdot 1.00212) +(M_{Q_{4+h,5+h}}(\theta)+\log\log T)\\&+\left(\frac{(2^{h+4}-1)(\sigma_{5+h}-\sigma)}{(2^{h+4}-2)(\sigma_{5+h}-\sigma_{h+4})}+\frac{(2^{h+5}-1)(\sigma-\sigma_{h+4})}{(2^{h+5}-2)(\sigma_{5+h}-\sigma_{h+4})}\right)\left(\frac{L_{Q_{4+h,5+h}}(\theta)}{2}+\log T\right).
\end{align*}
\item If $1/2\le \sigma\le \sigma_4$, we define
\begin{align*}
    F_{c,r}(\theta)&= \frac{(\sigma_{4}-\sigma)}{\sigma_{4}-\frac{1}{2}}\log( k_1\cdot 1.00212^{k_3})+\frac{(\sigma-\frac{1}{2})}{\sigma_{4}-\frac{1}{2}}\log(1.546\cdot 1.00212) \\&+\left(\frac{(k_2+1)(\sigma_{4}-\sigma)}{(\sigma_{4}-\frac{1}{2})}+\frac{15(\sigma-\frac{1}{2})}{14(\sigma_{4}-\frac{1}{2})}\right)\left(\frac{L_{Q_2}(\theta)}{2}+\log T\right)\\&+\left(\frac{k_3(\sigma_{4}-\sigma)}{\sigma_{4}-\frac{1}{2}}+\frac{(\sigma-\frac{1}{2})}{(\sigma_{4}-\frac{1}{2})}\right)(M_{Q_2}(\theta)+\log\log T).
\end{align*}
\item For $0\le\sigma\le 1/2$, as in \cite{HASANALIZADE2022219}, one has
\begin{equation*}
\begin{aligned}
F_{c, r}(\theta) & =(1-2 \sigma) \log \left(\frac{c_1\cdot 1.00212^{c_2}}{\sqrt{2 \pi}}\right)+2 \sigma \log (k_1\cdot 1.00212^{k_3})+\frac{1}{2} L_{-1}(\theta)+\log T \\
& +\frac{1-2 \sigma+4 k_2 \sigma}{2}\left(\frac{L_{Q_{11}}(\theta)}{2}+\log T\right)  +\left(c_2(1-2 \sigma)+2 k_3 \sigma\right)\left(M_{Q_{11}}(\theta)+\log \log T\right) .
\end{aligned}
\end{equation*}
\item If $-\eta\le\sigma\le 0$, then as in \cite{HASANALIZADE2022219}, we have
\begin{equation*}
\begin{aligned}
F_{c, r}(\theta) & =-\frac{\sigma}{\eta} \log \left(\frac{1+\eta}{c_1(2 \pi)^\eta}\right)+\log \left(\frac{c_1\cdot 1.00212^{c_2}}{\sqrt{2 \pi}}\right)+\frac{1}{2} L_{-1}(\theta)+\log T \\
& +\left(-\frac{\sigma(1+2 \eta)}{2 \eta}+\frac{\sigma+\eta}{2 \eta}\right)\left(\frac{L_{Q_{10}}(\theta)}{2}+\log T\right)+\frac{\sigma+\eta}{\eta} c_2\left(M_{Q_{10}}(\theta)+\log \log T\right) .
\end{aligned}
\end{equation*}
\item If $\sigma\le -\eta$ then as in \cite{HASANALIZADE2022219}, one has
\begin{equation*}
\begin{aligned}
F_{c, r}(\theta) & =\log \zeta(1-\sigma)+\frac{1}{2} L_{-1}(\theta)+\left(1+\frac{1-2 \sigma}{2}\right) \log T-\frac{1-2 \sigma}{2} \log 2 \pi \\
& +\frac{(1-2 \sigma+2[\sigma])}{4} L_{1-[\sigma]}(\theta)+\frac{1}{2} \sum_{j=1}^{-[\sigma]} L_{j-1}(\theta).
\end{aligned}
\end{equation*}
\end{itemize}

\subsection{Conclusion}
From \eqref{beforeintegral} and \eqref{boundhcrt}, following \cite{HASANALIZADE2022219}, by Lemma \ref{prop3.1}, we obtain
\begin{equation*}
    \begin{aligned}
         \left|N_{\mathbb{Q}}(T)-\frac{T}{\pi} \log \left(\frac{T}{2 \pi e}\right)\right|&\le \frac{7}{4}+\frac{1}{2}+\frac{1}{25 T}+\frac{2}{\pi} \log \zeta\left(\sigma_1\right)+\frac{1}{\log (r /(c-1 / 2))} \log \frac{\zeta(c)}{\zeta(2 c)}\\&-\frac{1}{\log (r /(c-1 / 2))} \log T+\frac{1}{\pi \log (r /(c-1 / 2))} \int_0^\pi F_{c, r}(\theta) d \theta\\&+\frac{E(T, \delta)}{\pi}.
    \end{aligned}
\end{equation*}
Hence, recalling $N_{\mathbb{Q}}(T)=2N(T)$, we obtain
\begin{equation*}
    \begin{aligned}
         \left|N(T)-\frac{T}{2\pi} \log \left(\frac{T}{2 \pi e}\right)\right|&\le \frac{7}{8}+\frac{1}{4}+\frac{1}{50 T}+\frac{1}{\pi} \log \zeta\left(\sigma_1\right)+\frac{1}{2\log (r /(c-1 / 2))} \log \frac{\zeta(c)}{\zeta(2 c)}\\&-\frac{1}{2\log (r /(c-1 / 2))} \log T+\frac{1}{2\pi \log (r /(c-1 / 2))} \int_0^\pi F_{c, r}(\theta) d \theta\\&+\frac{E(T, \delta)}{2\pi}.
    \end{aligned}
\end{equation*}
If we apply Lemma \ref{newprop3.1} instead, then we have
\begin{equation}\label{alt-bd-NT}
    \begin{aligned}
         \left|N(T)-\frac{T}{2\pi} \log \left(\frac{T}{2 \pi e}\right)\right|&\le \frac{7}{8}+\frac{1}{4}+\frac{1}{50 T}+\frac{1}{\pi} \log \zeta\left(\sigma_1\right)+\frac{\log(b\log T)}{2\log (r /(c-1 / 2))} \\&-\frac{1}{2\log (r /(c-1 / 2))} \log T+\frac{1}{2\pi \log (r /(c-1 / 2))} \int_0^\pi F_{c, r}(\theta) d \theta\\&+\frac{E(T, \delta)}{2\pi}.
    \end{aligned}
\end{equation}
Following \cite{Tr14-2,HASANALIZADE2022219}, we define 
\begin{equation}\label{deftheta}
\theta_y=\left\{\begin{array}{ll}
0 & \text { if } c+r \leq y; \\
\arccos \frac{y-c}{r} & \text { if } c-r \leq y \leq c+r ; \\
\pi & \text { if } y \leq c-r .
\end{array}\right.
\end{equation}
Now, we let $c, r$, and $\eta$ be positive real numbers satisfying\footnote{Note that $\theta_{-1/2}=\pi$. Indeed, by the definition of $\theta_y$, if we take $y=-1/2$, then $y\le c-r$ by the assumption \eqref{condcreta}.}
\begin{equation}\label{condcreta}
    -\frac{1}{2}<c-r<1-c<-\eta<1+\eta<c
\end{equation}
and $0<\eta \leq \frac{1}{2}$. To bound $\int_0^\pi F_{c, r}(\theta) d \theta$, we consider the splitting
$$
\int_0^\pi=\int_0^{\theta_{1+\eta}}+\int_{\theta_{1+\eta}}^{\theta_1}+\int_{\theta_1}^{\theta_{\sigma_{n+4}}}+\sum_{h=0}^{n-1}\int_{\sigma_{5+h}}^{\sigma_{4+h}}+\int_{\sigma_{4}}^{\theta_{\frac{1}{2}}}+\int_{\theta_{\frac{1}{2}}}^{\theta_0}+\int_{\theta_0}^{\theta_{-\eta}}+\int_{\theta_{-\eta}}^\pi.
$$
First, we recall two estimates \cite{HASANALIZADE2022219} for $L_{j}(\theta)$ defined in \eqref{Lj} and  $M_j(\theta)$ defined in \eqref{Mj} which hold for $T \geq T_0$ and $\theta \in[0, \pi]$. Defining, for $\theta \in[0, \pi]$, $$
L_j^{\star}(\theta)=\frac{1}{T_0}(j+c+r \cos \theta)^2+\frac{1}{T_0}(r \sin \theta)^2+2 r \sin \theta,
$$
we have
\begin{equation*}
L_j(\theta) \leq \frac{L_j^{\star}(\theta)}{T}
\end{equation*}
and
\begin{equation*}
M_j(\theta)\leq \frac{L_j^{\star}(\theta)}{2 T \log T}.
\end{equation*}
Now, we proceed with the estimate of each integral as in \cite{HASANALIZADE2022219} to derive
\begin{equation*}
\int_0^{\theta_{1+\eta}} F_{c, r}(\theta) d \theta \leq \log T \int_0^{\theta_{1+\eta}} 1 d \theta+\int_0^{\theta_{1+\eta}} \log \zeta(\sigma) d \theta+\frac{1}{2 T} \int_0^{\theta_{1+\eta}} L_{-1}^{\star}(\theta) d \theta,
\end{equation*}
\begin{equation*}
    \begin{aligned}
        &\int_{\theta_{1+\eta}}^{\theta_1} F_{c, r}(\theta) d \theta\\&\le \log T \int_{\theta_{1+\eta}}^{\theta_1} 1 d \theta+\frac{c_2}{\eta} \log \log T \int_{\theta_{1+\eta}}^{\theta_1}(1+\eta-\sigma) d \theta+\frac{\log (c_1\cdot 1.00212^{c_2})}{\eta} \int_{\theta_{1+\eta}}^{\theta_1}(1+\eta-\sigma) d \theta\\&+\frac{\log \zeta(1+\eta)}{\eta} \int_{\theta_{1+\eta}}^{\theta_1}(\sigma-1) d \theta+\frac{1}{2 T} \int_{\theta_{1+\eta}}^{\theta_1} L_{Q_0}^{\star}(\theta) d \theta+\frac{c_2}{2 \eta T \log T} \int_{\theta_{1+\eta}}^{\theta_1}(1+\eta-\sigma) L_{Q_0}^{\star}(\theta) d \theta,
    \end{aligned}
\end{equation*}
\begin{equation*}
    \begin{aligned}
        &\int_{\theta_{\frac{1}{2}}}^{\theta_0} F_{c, r}(\theta) d \theta\\&\le \log T \int_{\theta_{\frac{1}{2}}}^{\theta_0} 1 d \theta+\frac{\log T}{2} \int_{\theta_{\frac{1}{2}}}^{\theta_0} (1-2 \sigma+4 k_2 \sigma) d \theta+\log \log T \int_{\theta_{\frac{1}{2}}}^{\theta_0} (c_2(1-2 \sigma)+2 k_3 \sigma) d \theta\\&+\log \left(\frac{c_1\cdot 1.00212^{c_2}}{\sqrt{2 \pi}}\right) \int_{\theta_{\frac{1}{2}}}^{\theta_0} (1-2 \sigma) d \theta+2 \log (k_1\cdot 1.00212^{k_3}) \int_{\theta_{\frac{1}{2}}}^{\theta_0} \sigma d \theta+\frac{1}{2 T} \int_{\theta_{\frac{1}{2}}}^{\theta_0} L_{-1}^{\star}(\theta) d \theta\\&+\frac{1}{4 T} \int_{\theta_{\frac{1}{2}}}^{\theta_0}\left(1-2 \sigma+4 k_2 \sigma\right) L_{Q_{11}}^{\star}(\theta) d \theta+\frac{1}{2 T \log T} \int_{\theta_{\frac{1}{2}}}^{\theta_0}\left(c_2(1-2 \sigma)+2 k_3 \sigma\right) L_{Q_{11}}^{\star}(\theta) d \theta ,
    \end{aligned}
\end{equation*}
\begin{equation*}
    \begin{aligned}
        \int_{\theta_0}^{\theta_{-\eta}} F_{c, r}(\theta) d \theta &\le \log T \int_{\theta_0}^{\theta_{-\eta}} \left(1-\frac{\sigma(1+2 \eta)}{2 \eta}+\frac{\sigma+\eta}{2 \eta}\right) d \theta+\log \log T \int_{\theta_0}^{\theta_{-\eta}} \frac{\sigma+\eta}{\eta} c_2 d \theta\\&+\int_{\theta_0}^{\theta_{-\eta}}\left(-\frac{\sigma}{\eta} \log \left(\frac{1+\eta}{c_1(2 \pi)^\eta}\right)+\log \left(\frac{c_1\cdot 1.00212^{c_2}}{\sqrt{2 \pi}}\right)\right) d \theta+\frac{1}{2 T} \int_{\theta_0}^{\theta_{-\eta}} L_{-1}^{\star}(\theta) d \theta\\&+\frac{1}{T} \int_{\theta_0}^{\theta_{-\eta}}\left(-\frac{\sigma(1+2 \eta)}{4 \eta}+\frac{\sigma+\eta}{4 \eta}\right) L_{Q_{10}}^{\star}(\theta) d \theta+\frac{1}{2 T \log T} \int_{\theta_0}^{\theta_{-\eta}} \frac{\sigma+\eta}{\eta} c_2 L_{Q_{10}}^{\star}(\theta) d \theta
    \end{aligned}
\end{equation*}
and
\begin{equation*}
    \begin{aligned}
        \int_{\theta_{-\eta}}^\pi F_{c, r}(\theta) d \theta 
        &\leq \log T \int_{\theta_{-\eta}}^\pi 1+\frac{1-2 \sigma}{2} d \theta+\int_{\theta_{-\eta}}^\pi \log \zeta(1-\sigma) d \theta-\log 2 \pi \int_{\theta_{-\eta}}^\pi \frac{1-2 \sigma}{2} d \theta\\
        &+\frac{1}{2 T} \int_{\theta_{-\eta}}^\pi L_{-1}^{\star}(\theta) d \theta+\frac{1}{T} \int_{\theta_{-\eta}}^{\theta_{-\frac{1}{2}}} \frac{1-2 \sigma}{4} L_1^{\star}(\theta) d \theta\\
        &+\sum_{j=1}^{\infty} \int_{\theta_{-j+\frac{1}{2}}}^{\theta_{-j-\frac{1}{2}}}\left(\frac{1-2 \sigma-2 j}{4} L_{j+1}(\theta)+\frac{1}{2} \sum_{k=1}^j L_{k-1}(\theta)\right) d \theta.
    \end{aligned}
\end{equation*}
Observe that
\begin{equation*}
    \begin{aligned}
        &\int_{\theta_1}^{\theta_{\sigma_{n+4}}} F_{c, r}(\theta) d \theta\\&\le \frac{\log T}{(2^{n+4}-2)(1-\sigma_{n+4})}\int_{\theta_1}^{\theta_{\sigma_{n+4}}}\left((2^{n+4}-1)(1-\sigma)+(2^{n+4}-2)(\sigma-\sigma_{n+4})\right)d\theta\\&+\frac{\log\log T}{1-\sigma_{n+4}} \int_{\theta_1}^{\theta_{\sigma_{n+4}}}\left((1-\sigma)+c_2(\sigma-\sigma_{n+4})\right)d\theta\\&+\frac{\log(1.546\cdot 1.00212)}{(1-\sigma_{n+4})}\int_{\theta_1}^{\theta_{\sigma_{n+4}}}(1-\sigma)d\theta+\frac{\log (c_1\cdot 1.00212^{c_2})}{(1-\sigma_{n+4})}\int_{\theta_1}^{\theta_{\sigma_{n+4}}}(\sigma-\sigma_{n+4})d\theta\\&+\frac{1}{2T(2^{n+4}-2)(1-\sigma_{n+4})}\int_{\theta_1}^{\theta_{\sigma_{n+4}}}\left((2^{n+4}-1)(1-\sigma)+(2^{n+4}-2)(\sigma-\sigma_{n+4})\right)L_{Q_{0,n+4}}^{\star}(\theta)d\theta\\&+\frac{1}{2T\log T(1-\sigma_{n+4})} \int_{\theta_1}^{\theta_{\sigma_{n+4}}}\left((1-\sigma)+c_2(\sigma-\sigma_{n+4})\right)L_{Q_{0,n+4}}^{\star}(\theta)d\theta.
    \end{aligned}
\end{equation*}
Hence, for every $0\le h\le n-1$, we have
\begin{equation*}
    \begin{aligned}
&\int_{\theta_{\sigma_{5+h}}}^{\theta_{\sigma_{4+h}}} F_{c, r}(\theta) d \theta\\&\le \log\log T\int_{\theta_{\sigma_{5+h}}}^{\theta_{\sigma_{4+h}}} 1 d \theta+\frac{\log T}{(2^{h+4}-2)(2^{h+5}-2)(\sigma_{5+h}-\sigma_{h+4})}\\& \times \int_{\theta_{\sigma_{5+h}}}^{\theta_{\sigma_{4+h}}}\left((2^{h+5}-2)(2^{h+4}-1)(\sigma_{5+h}-\sigma)+(2^{h+4}-2)(2^{h+5}-1)(\sigma-\sigma_{h+4})\right)d\theta\\&+\log(1.546\cdot 1.00212)\int_{\theta_{\sigma_{5+h}}}^{\theta_{\sigma_{4+h}}}1d\theta+\frac{1}{2T\log T}\int_{\theta_{\sigma_{5+h}}}^{\theta_{\sigma_{4+h}}}L_{Q_{4+h,5+h}}^{\star}(\theta)d\theta\\&+\frac{1}{2T(2^{h+4}-2)(2^{h+5}-2)(\sigma_{5+h}-\sigma_{h+4})}\\
      & \times \int_{\theta_{\sigma_{5+h}}}^{\theta_{\sigma_{4+h}}}\left((2^{h+5}-2)(2^{h+4}-1)(\sigma_{5+h}-\sigma)+(2^{h+4}-2)(2^{h+5}-1)(\sigma-\sigma_{h+4})\right)L_{Q_{4+h,5+h}}^{\star}(\theta)d\theta.
    \end{aligned}
\end{equation*}
Finally, we have
\begin{equation*}
    \begin{aligned}
        \int_{\theta_{\sigma_4}}^{\theta_{1/2}}F_{c, r}(\theta) d \theta
        &\le \frac{\log T}{14(\sigma_{4}-\frac{1}{2})}\int_{\theta_{\sigma_{4}}}^{\theta_{{1/2}}}\left(14(k_2+1)(\sigma_{4}-\sigma)+15\left(\sigma-\frac{1}{2}\right)\right)d\theta\\&+ \frac{\log\log T}{(\sigma_{4}-\frac{1}{2})}\int_{\theta_{\sigma_{4}}}^{\theta_{{1/2}}}\left(k_3(\sigma_{4}-\sigma)+\left(\sigma-\frac{1}{2}\right)\right)d\theta\\&+ \frac{\log (k_1\cdot 1.00212^{k_3})}{\sigma_{4}-\frac{1}{2}}\int_{\theta_{\sigma_4}}^{\theta_{1/2}}(\sigma_{4}-\sigma)d\theta+\frac{\log(1.546\cdot 1.00212)}{\sigma_{4}-\frac{1}{2}}\int_{\theta_{\sigma_4}}^{\theta_{1/2}}\left(\sigma-\frac{1}{2}\right)d\theta\\&+\frac{1}{2T14(\sigma_{4}-\frac{1}{2})}\int_{\theta_{\sigma_{4}}}^{\theta_{{1/2}}}\left(14(k_2+1)(\sigma_{4}-\sigma)+15\left(\sigma-\frac{1}{2}\right)\right)L_{Q_{2}}^{\star}(\theta)d\theta\\&+ \frac{1}{2T\log T(\sigma_{4}-\frac{1}{2})}\int_{\theta_{\sigma_{4}}}^{\theta_{{1/2}}}\left(k_3(\sigma_{4}-\sigma)+\left(\sigma-\frac{1}{2}\right)\right)L_{Q_{2}}^{\star}(\theta)d\theta.
    \end{aligned}
\end{equation*}

With the above estimates in hand, we are ready to estimate the constants.
\subsubsection{Constant $C_1$}
Setting
\begin{align*}
    \overline{C}_1 &=\frac{1}{(2^{n+4}-2)(1-\sigma_{n+4})}\int_{\theta_1}^{\theta_{\sigma_{n+4}}}\left((2^{n+4}-1)(1-\sigma)+(2^{n+4}-2)(\sigma-\sigma_{n+4})\right)d\theta\\&+\sum_{h=0}^{n-1}\frac{1}{(2^{h+4}-2)(2^{h+5}-2)(\sigma_{5+h}-\sigma_{h+4})}\\
    & \times \int_{\theta_{\sigma_{5+h}}}^{\theta_{\sigma_{4+h}}}\left((2^{h+5}-2)(2^{h+4}-1)(\sigma_{5+h}-\sigma)+(2^{h+4}-2)(2^{h+5}-1)(\sigma-\sigma_{h+4})\right)d\theta\\&+\frac{1}{14(\sigma_{4}-\frac{1}{2})}\int_{\theta_{\sigma_{4}}}^{\theta_{{1/2}}}\left(14(k_2+1)(\sigma_{4}-\sigma)+15\left(\sigma-\frac{1}{2}\right)\right)d\theta\\&+ \frac{1}{2} \int_{\theta_{\frac{1}{2}}}^{\theta_0}( 1-2 \sigma+4 k_2 \sigma) d \theta+\int_{\theta-\eta}^\pi \frac{1-2 \sigma}{2} d \theta+(\theta_1-\theta_{1/2})
    +\int_{\theta_0}^{\theta_{-\eta}} \left(-\frac{\sigma(1+2 \eta)}{2 \eta}+\frac{\sigma+\eta}{2 \eta}\right) d \theta,
\end{align*}
we can express $C_1$ as
\begin{equation}\label{C1}
C_1=\frac{\overline{C}_1}{2 \pi \log (r /(c-1 / 2))}.
\end{equation}
\subsubsection{Constants $C_2$ and $C'_2$}
Given
\begin{align*}
    \overline{C}_2&=\frac{1}{1-\sigma_{n+4}} \int_{\theta_1}^{\theta_{\sigma_{n+4}}}\left((1-\sigma)+c_2(\sigma-\sigma_{n+4})\right)d\theta\\&+\sum_{h=0}^{n-1}\int_{\theta_{\sigma_{5+h}}}^{\theta_{\sigma_{4+h}}}1d\theta+ \frac{1}{(\sigma_{4}-\frac{1}{2})}\int_{\theta_{\sigma_{4}}}^{\theta_{{1/2}}}\left(k_3(\sigma_{4}-\sigma)+\left(\sigma-\frac{1}{2}\right)\right)d\theta\\&+\frac{c_2}{\eta} \int_{\theta_{1+\eta}}^{\theta_1}(1+\eta-\sigma) d \theta++\int_{\theta_{\frac{1}{2}}}^{\theta_0} (c_2(1-2 \sigma)+2 k_3 \sigma) d \theta+\int_{\theta_0}^{\theta_{-\eta}} \frac{\sigma+\eta}{\eta} c_2 d \theta\\&=\frac{1}{1-\sigma_{n+4}} \int_{\theta_1}^{\theta_{\sigma_{n+4}}}\left((1-\sigma)+c_2(\sigma-\sigma_{n+4})\right)d\theta+ \theta_{\sigma_4}\\&-\theta_{\sigma_{n+4}}+ \frac{1}{(\sigma_{4}-\frac{1}{2})}\int_{\theta_{\sigma_{4}}}^{\theta_{{1/2}}}\left(k_3(\sigma_{4}-\sigma)+\left(\sigma-\frac{1}{2}\right)\right)d\theta\\&+\frac{c_2}{\eta} \int_{\theta_{1+\eta}}^{\theta_1}(1+\eta-\sigma) d \theta++\int_{\theta_{\frac{1}{2}}}^{\theta_0} (c_2(1-2 \sigma)+2 k_3 \sigma) d \theta+\int_{\theta_0}^{\theta_{-\eta}} \frac{\sigma+\eta}{\eta} c_2 d \theta,
    \end{align*}
 by Lemma \ref{prop3.1}, our $C_2$ can be written as
\begin{equation}\label{C2}
C_2=\frac{\overline{C}_2}{2 \pi \log (r /(c-1 / 2))};
\end{equation}
while Lemma \ref{newprop3.1} yields that
\begin{equation}\label{C2p}
C'_2=\frac{\overline{C}_2}{2 \pi \log (r /(c-1 / 2))}+\frac{2.00204}{2\log (r /(c-1 / 2))}.
\end{equation}

\begin{remark}
  To establish \eqref{C2p}, we used the bound $\log(b\log T)\le B\log\log T$ for some  $B$ depending only on $b$ and our range of $T$. One may instead use $\log(b \log T )= \log b + \log\log T$, which would lead to slightly smaller $C_2'$ at the expense of $C'_3$.  
\end{remark}

\subsubsection{Constants $C_3$ and $C'_3$}
By Lemma \ref{prop3.1}, $C_3$ can be expressed as
\begin{equation}\label{C3}
 \begin{aligned}
    C_3&=\frac{7}{8}+\frac{1}{4}+\frac{1}{50T}+\frac{1}{\pi }\log\zeta(\sigma_1)+\frac{1}{2\log (r /(c-1 / 2))} \log \frac{\zeta(c)}{\zeta(2 c)}+\frac{1}{2}\left(\frac{640 \delta-112}{1536\left(3 T_0-1\right)}+\frac{1}{2^{10}}\right)\\
    &+\frac{1}{2 \pi \log (r /(c-1 / 2))}\left(D_3+\kappa_1\left(J_1\right)+\kappa_2\left(J_2\right)+\kappa_3\left(T_0\right)\right).
\end{aligned}   
\end{equation}
In addition, Lemma \ref{newprop3.1} yields
\begin{equation}\label{C3p}
    \begin{aligned}
    C'_3&=\frac{7}{8}+\frac{1}{4}+\frac{1}{50T}+\frac{1}{\pi }\log\zeta(\sigma_1)+\frac{1}{2}\left(\frac{640 \delta-112}{1536\left(3 T_0-1\right)}+\frac{1}{2^{10}}\right)\\&+\frac{1}{2 \pi \log (r /(c-1 / 2))}\left(D_3+\kappa_1\left(J_1\right)+\kappa_2\left(J_2\right)+\kappa_3\left(T_0\right)\right),
\end{aligned}
\end{equation}
where
\begin{align*}
    D_3&=\frac{\log (c_1\cdot 1.00212^{c_2})}{\eta} \int_{\theta_{1+\eta}}^{\theta_1}(1+\eta-\sigma) d \theta+\frac{\log \zeta(1+\eta)}{\eta} \int_{\theta_{1+\eta}}^{\theta_1}(\sigma-1) d \theta\\
    & +\frac{\log(1.546\cdot 1.00212)}{(1-\sigma_{n+4})}\int_{\theta_1}^{\theta_{\sigma_{n+4}}}(1-\sigma)d\theta+\frac{\log (c_1\cdot 1.00212^{c_2})}{(1-\sigma_{n+4})}\int_{\theta_1}^{\theta_{\sigma_{n+4}}}(\sigma-\sigma_{n+4})d\theta\\
    &+\sum_{h=0}^{n-1}\log(1.546\cdot 1.00212)\int_{\theta_{\sigma_{5+h}}}^{\theta_{\sigma_{4+h}}}d\theta + \frac{\log( k_1\cdot 1.00212^{k_3})}{\sigma_{4}-\frac{1}{2}}\int_{\theta_{\sigma_4}}^{\theta_{1/2}}(\sigma_{4}-\sigma)d\theta\\
    &+\frac{\log(1.546\cdot 1.00212)}{\sigma_{4}-\frac{1}{2}}\int_{\theta_{\sigma_4}}^{\theta_{1/2}}\left(\sigma-\frac{1}{2}\right)d\theta+\log \left(\frac{c_1\cdot 1.00212^{c_2}}{\sqrt{2 \pi}}\right) \int_{\theta_{\frac{1}{2}}}^{\theta_0} (1-2 \sigma) d \theta\\&+2 \log (k_1 \cdot 1.00212^{k_3})\int_{\theta_{\frac{1}{2}}}^{\theta_0} \sigma d \theta+\int_{\theta_0}^{\theta_{-\eta}}\left(-\frac{\sigma}{\eta} \log \left(\frac{1+\eta}{c_1(2 \pi)^\eta}\right)+\log \left(\frac{c_1\cdot 1.00212^{c_2}}{\sqrt{2 \pi}}\right)\right) d \theta\\&-(\log 2 \pi) \int_{\theta_{-\eta}}^\pi \frac{1-2 \sigma}{2} d \theta+\frac{\log \zeta(1+\eta)+\log \zeta(c)}{2}\left(\theta_{1+\eta}-\frac{\pi}{2}\right)+\frac{\pi}{4 J_1} \log \zeta(c)\\&+\frac{\log \zeta(1+\eta)+\log \zeta(c)}{2}\left(\theta_{1-c}-\theta_{-\eta}\right)+\frac{\pi-\theta_{1-c}}{2 J_2} \log \zeta(c),
\end{align*}
\[
\kappa_1\left(J_1\right)=\frac{\pi}{4 J_1}\left(\log \zeta(c+r)+2 \sum_{j=1}^{J_1-1} \log \zeta\left(c+r \cos \frac{\pi j}{2 J_1}\right)\right),
\]
\begin{equation*}
\begin{aligned}
\kappa_2\left(J_2\right)= & \frac{\pi-\theta_{1-c}}{2 J_2}(\log \zeta(1-c+r)  +2 \sum_{j=1}^{J_2-1} \log \zeta\left(1-c-r \cos \left(\frac{\pi j}{J_2}+\left(1-\frac{j}{J_2}\right) \theta_{1-c}\right)\right),
\end{aligned}
\end{equation*}
and
\begin{align*}
    \kappa_3(T_0)=
\frac{1}{2 T_0} \max \left\{0,\mathcal{M}_1\right\}+\frac{1}{2T_0\log\log T_0}\max \left\{0,\mathcal{M}_2\right\},
\end{align*}
with 
\begin{align*}
    &\mathcal{M}_1=\int_0^{\theta_{1+\eta}} L_{-1}^{\star}(\theta)+\int_{\theta_{1+\eta}}^{\theta_1} L_{Q_0}^{\star}(\theta) d \theta+\int_{\theta_{\frac{1}{2}}}^{\theta_0} \left(L_{-1}^{\star}(\theta) +\frac{1}{2} \left(1-2 \sigma+4 k_2 \sigma\right) L_{Q_{11}}^{\star}(\theta)\right) d \theta\\
&+\frac{1}{(2^{n+4}-2)(1-\sigma_{n+4})}\int_{\theta_1}^{\theta_{\sigma_{n+4}}}\left((2^{n+4}-1)(1-\sigma)+(2^{n+4}-2)(\sigma-\sigma_{n+4})\right)L_{Q_{0,n+4}}^{\star}(\theta)d\theta\\
&+\sum_{h=0}^{n-1}\frac{1}{(2^{h+4}-2)(2^{h+5}-2)(\sigma_{5+h}-\sigma_{h+4})}\\
& \times \int_{\theta_{\sigma_{5+h}}}^{\theta_{\sigma_{4+h}}}\left((2^{h+5}-2)(2^{h+4}-1)(\sigma_{5+h}-\sigma)+(2^{h+4}-2)(2^{h+5}-1)(\sigma-\sigma_{h+4})\right)L_{Q_{4+h,5+h}}^{\star}(\theta)d\theta\\
&+\frac{1}{14(\sigma_{4}-\frac{1}{2})}\int_{\theta_{\sigma_{4}}}^{\theta_{{1/2}}}\left(14(k_2+1)(\sigma_{4}-\sigma)+15\left(\sigma-\frac{1}{2}\right)\right)L_{Q_{2}}^{\star}(\theta)d\theta\\
&+\int_{\theta_0}^{\theta_{-\eta}}\left( L_{-1}^{\star}(\theta)+\left(-\sigma+\frac{1}{2}\right) L_{Q_{10}}^{\star}(\theta)\right) d \theta +\int_{\theta_{-\eta}}^\pi\left( L_{-1}^{\star}(\theta)+\frac{1-2 \sigma}{2} L_1^{\star}(\theta)\right) d \theta
\end{align*}
and
\begin{align*}
    &\mathcal{M}_2=\int_{\theta_{1+\eta}}^{\theta_1} \frac{c_2}{\eta}(1+\eta-\sigma) L_{Q_0}^{\star}(\theta) d \theta +\frac{1}{1-\sigma_{n+4}} \int_{\theta_1}^{\theta_{\sigma_{n+4}}}\left((1-\sigma)+c_2(\sigma-\sigma_{n+4})\right)L_{Q_{0,n+4}}^{\star}(\theta)d\theta\\
&+\sum_{h=0}^{n-1}\int_{\theta_{\sigma_{5+h}}}^{\theta_{\sigma_{4+h}}}L_{Q_{4+h,5+h}}^{\star}(\theta)d\theta + \frac{1}{(\sigma_{4}-\frac{1}{2})}\int_{\theta_{\sigma_{4}}}^{\theta_{{1/2}}}\left(k_3(\sigma_{4}-\sigma)+\left(\sigma-\frac{1}{2}\right)\right)L_{Q_{2}}^{\star}(\theta)d\theta\\
&+\int_{\theta_{\frac{1}{2}}}^{\theta_0}\left(c_2(1-2 \sigma)+2 k_3 \sigma\right) L_{Q_{11}}^{\star}(\theta) d \theta+\int_{\theta_0}^{\theta_{-\eta}} \frac{\sigma+\eta}{\eta} c_2 L_{Q_{10}}^{\star}(\theta) d \theta.
\end{align*}
Hence, we complete the proof of Theorem \ref{generalth1}.

\section{Proof of Theorem \ref{th2}}
 We begin by recalling that
\begin{equation*}
S(T)=\frac{1}{\pi} \Delta_{\mathcal{C}_0} \arg \zeta(s)=\frac{1}{\pi} \Delta_{\mathcal{C}_1} \arg \zeta(s)+\frac{1}{\pi} \Delta_{\mathcal{C}_2} \arg (s-1) \zeta(s)-\frac{1}{\pi} \Delta_{\mathcal{C}_2} \arg (s-1),
\end{equation*}
where
\begin{equation*}
\left|\Delta_{\mathcal{C}_1} \arg \zeta(s)\right| \leq \log \zeta\left(\sigma_1\right)
\end{equation*}
and
\begin{equation*}
\left|\Delta_{\mathcal{C}_2} \arg (s-1)\right|=\arctan \left(\frac{\sigma_1-1}{T}\right)+\arctan \left(\frac{1}{2 T}\right) \leq \arctan \left(\frac{\sigma_1-1}{T_0}\right)+\arctan \left(\frac{1}{2 T_0}\right)
\end{equation*}
for $T\ge T_0$. By \eqref{beforeintegral} and Theorem \ref{generalth1}, we know
\begin{equation*}
    \begin{aligned}
        \left|N(T)-\frac{T}{2\pi} \log \left(\frac{T}{2 \pi e}\right)\right|&\le \frac{7}{8}+ \frac{1}{50T}+\frac{1}{\pi} \log \zeta\left(\sigma_1\right)+\frac{1}{\pi}\left|\Delta_{\mathcal{C}_2} \arg ((s-1) \zeta(s))\right|\\&\le C_1\log T +\min\{C_2\log\log T+C_3,C'_2\log\log T+C'_3\}
    \end{aligned}
\end{equation*}
and hence
\begin{align*}
    &\frac{1}{\pi}\left|\Delta_{\mathcal{C}_2} \arg ((s-1) \zeta(s))\right|\\&\le C_1\log T+\min\{C_2\log\log T+C_3,C'_2\log\log T+C'_3\}-\frac{7}{8}-\frac{1}{50T}-\frac{1}{\pi} \log \zeta\left(\sigma_1\right).
\end{align*}
Therefore, for $T\ge T_0$, the following estimate holds:
\begin{equation*}
    |S(T)|\le C_1\log T+\min\{C_2\log\log T+\tilde{C_3},C'_2\log\log T+\tilde{C'_3}\},
\end{equation*}
where
\begin{equation}\label{tildeC3}
  \tilde{C_3}=C_3-\frac{7}{8}-\frac{1}{50T}+\frac{1}{\pi}\left(\arctan \left(\frac{\sigma_1-1}{T_0}\right)+\arctan \left(\frac{1}{2 T_0}\right)\right)  
\end{equation}
and
\begin{equation}\label{TildeC3p}
    \tilde{C'_3}=C'_3-\frac{7}{8}-\frac{1}{50T}+\frac{1}{\pi}\left(\arctan \left(\frac{\sigma_1-1}{T_0}\right)+\arctan \left(\frac{1}{2 T_0}\right)\right).
\end{equation}
This completes the proof of Theorem \ref{th2}.

\section{Proofs of Theorem \ref{th1} and Corollary \ref{corst}}
First, we apply Theorem \ref{generalth1} and Theorem \ref{th2} with $T_0=30610046000$. Furthermore, we take $J_1=64$, $J_2=39$ and we choose the  parameters $Q_i$ as
\[
(Q_0,Q_1,Q_2,Q_3,Q_4,Q_5,Q_6,Q_7,Q_8,Q_9,Q_{10},Q_{11})=(1,1.18,1.18,3.9,1,1,1,1,1,1,2.3,3.9)
\]
to establish the following:
 \begin{table}[h]
\def\arraystretch{1.3}
\centering
        \begin{tabular}{|c|c|c|c|c|c|c|c|c|c|}
        \hline
    $c$   & $r$ & $\eta$ & $C_1$ & $C_2$ & $C_2'$& $C_3$& $C_3'$ & $\Tilde{C_3}$ & $\Tilde{C'_3}$ \\ \hline
     $1.000225$  & $1.000605$ & $0.000158$ & $0.10076$ & $0.24460$ & $1.68845$ & $8.08344$ & $2.38456$ & $7.20844$ & $1.50956$ \\\hline  
     $1.070007$ & $1.182997$ & $0.069901$ & $0.11000$ & $0.17447$ & $1.54543$ & $3.71067$ & $2.15392$ & $2.83567$ & $1.27892$\\\hline
     $1.043400$ & $1.250450 $ & $ 0.040000$ & $0.11200$ & $0.12567$ & $ 1.32678$ & $ 3.77417$ & $2.14783$ & $2.89916$ & $ 1.27283$\\\hline
     $1.000060$ & $1.499556$ & $1.542440\cdot 10^{-5}$ & $0.12355$ & $0.06782$ & $ 0.97933$ & $6.25796$ & $2.05854$ & $5.38296$ & $1.18354$\\\hline
     $1.499159$ & $1.998357$ & $0.499050 $ & $0.16732$ & $0.17266$ & $1.61679$ & $1.96334$ & $1.40271$ & $1.08834$& $0.52771$ \\\hline
  \end{tabular} \caption{Some admissible values for $C_1,C_2,C'_2,C_3,C'_3, $ $\Tilde{C_3}$, $\Tilde{C'_3}$}\label{table1} 
  \end{table}

Note that the first and third rows of Table \ref{table1} give Theorem \ref{th1} and Corollary \ref{corst} for $T\ge 30610046000$. In addition, the second row of Table \ref{table1} improves on \cite{PLATT2015842} for all $T\ge \exp(113)$.\footnote{However, unfortunately, the proof in \cite{PLATT2015842} relies on \cite{Tr14-2}, which contains an error as pointed out earlier.}

  Finally, for $e\le T\le 30610046000$, in order to estimate $|S(T)|$, we use the known bound \eqref{plattbound} computed by Platt, which is in particular sharper than
  \[0.10076\log T+\min\{0.24460\log\log T+7.20844, 1.68845\log\log T+1.50956\}\]
  for every $e\le T\le 30610046000$. Hence, Corollary \ref{corst} follows by combining the two cases.
  
  Finally, it remains to prove Theorem \ref{th1} for $e\le T\le 30610046000$. As in \cite{HASANALIZADE2022219}, since
\begin{equation}\label{NT-exp}
N(T)=S(T)+\frac{T}{2\pi} \log \left(\frac{T}{2 \pi e}\right)+\frac{7}{8}+\frac{g(T)}{2},
\end{equation}
for $e\le T\le 30610046000$, we have
  \begin{equation*}
\left|N(T)-\frac{T}{2 \pi} \log \left(\frac{T}{2 \pi e}\right)\right| \leq|S(T)|+\frac{1}{2}|g(T)|+\frac{7}{8}\leq 2.5167+\frac{1}{50 e}+\frac{7}{8},
\end{equation*}
which is always smaller than 
\[
0.10076\log T+0.24460\log\log T+8.08344
\]
for every $e\le T\le 30610046000$. Therefore, by combining the two cases, Theorem \ref{th1} follows.

\section{Proof of Corollary \ref{cor1} }\label{proofcor12}
In this section, we will prove Corollary \ref{cor1}. It follows from \eqref{NT-exp} that
\begin{equation}\label{eq:t}
    \begin{aligned}
        &N(T+1)-N(T)\\
        &=S(T+1)-S(T)+\frac{T+1}{2\pi} \log \left(\frac{T+1}{2 \pi e}\right)-\frac{T}{2\pi} \log \left(\frac{T}{2 \pi e}\right)+\frac{g(T+1)-g(T)}{2}
    \end{aligned}
\end{equation}
and
\begin{equation}\label{eq:t-1}
    \begin{aligned}
        &N(T+1)-N(T-1)\\&=S(T+1)-S(T-1)+\frac{T+1}{2\pi} \log \left(\frac{T+1}{2 \pi e}\right)-\frac{T-1}{2\pi} \log \left(\frac{T-1}{2 \pi e}\right)+\frac{g(T+1)-g(T-1)}{2}.
    \end{aligned}
\end{equation}
Writing
\[
\log \left(\frac{T+1}{2 \pi e}\right)=\log \left(\frac{T}{2 \pi e}\right)+\log \left(1+\frac{1}{T}\right)
\quad\text{and}\quad
\log \left(\frac{T-1}{2 \pi e}\right)=\log \left(\frac{T}{2 \pi e}\right)+\log \left(1-\frac{1}{T}\right),
\]
and using the Taylor expansions at $x=0$ and $y=\infty$
\begin{align*}
     \log \left(1+x\right)=x-\frac{x^2}{2}+\frac{x^3}{3}-\frac{x^4}{4}+O(x^5)
     \quad\text{and}\quad
     \log\left(\frac{1+y}{-1+y}\right)=\frac{2}{y}+\frac{2}{3y^3}+\frac{2}{5y^5}+O\left(\frac{1}{y^6}\right)
\end{align*}
with $x=1/T$, $x= -1/T$, $y=T$, we derive that, for $T\ge 2$,
\begin{equation}\label{upper1}
    \begin{aligned}
        &\frac{T+1}{2\pi} \log \left(\frac{T+1}{2 \pi e}\right)-\frac{T}{2\pi} \log \left(\frac{T}{2 \pi e}\right)\\
        &=\frac{T}{2\pi}\log \left(1+\frac{1}{T}\right)+\frac{1}{2\pi }\left(\log \left(\frac{T}{2 \pi e}\right)+\log \left(1+\frac{1}{T}\right)\right)\\
        &=\frac{1}{2\pi}\left(1-\frac{1}{2T}+\frac{1}{3T^2}-\cdots\right)+\frac{1}{2\pi}\left(\frac{1}{T}-\frac{1}{2T^2}+\frac{1}{3T^3}-\cdots\right)+\frac{1}{2\pi}\log T-\frac{1}{2\pi}\log(2\pi e)\\
        &<\frac{1}{2\pi}\log T+\frac{3}{4\pi}-\frac{1}{2\pi}\log(2\pi e),
        \end{aligned}
\end{equation}
and, similarly, 
\begin{equation}\label{upper2}
    \begin{aligned}
        &\frac{T+1}{2\pi} \log \left(\frac{T+1}{2 \pi e}\right)-\frac{T-1}{2\pi} \log \left(\frac{T-1}{2 \pi e}\right)\\
        &=\frac{T}{2\pi}\left(\log \left(1+\frac{1}{T}\right)-\log \left(1-\frac{1}{T}\right)\right)+\frac{1}{2\pi }\left(2\log \left(\frac{T}{2 \pi e}\right)+\log \left(1+\frac{1}{T}\right)+\log \left(1-\frac{1}{T}\right)\right)\\
        &<\frac{1}{\pi}\log T+\frac{\log 3}{\pi}-\frac{1}{\pi}\log(2\pi e).
        \end{aligned}
\end{equation}
Furthermore, we observe that
\begin{equation*}
    \begin{aligned}
        \left|\frac{g(T+1)-g(T)}{2}\right|&\le \max\{|g(T+1)|,|g(T)|\}\le \frac{1}{25T}, \\
        \left|\frac{g(T+1)-g(T-1)}{2}\right|&\le \max\{|g(T+1)|,|g(T-1)|\}\le \frac{1}{25(T-1)}.
    \end{aligned}
\end{equation*}
Finally,  if $T+1> 30610046000$, then Theorem \ref{th2} implies that
\begin{equation*}
    \begin{aligned}
       |S(T+1)-S(T)|&\le 2.00001C_1\log T+2.00001\min\{C_2\log\log T +\Tilde{C_3},C'_2\log\log T +\Tilde{C'_3}\}
\\
     |S(T+1)-S(T-1)|&\le 2.00001C_1\log T+2.00001\min\{C_2\log\log T +\Tilde{C_3},C'_2\log\log T +\Tilde{C'_3}\}.
    \end{aligned}
\end{equation*}
For $3\le T+1\le 30610046000$, by \eqref{plattbound}, we obtain
\begin{equation*}
    \begin{aligned}
  |S(T+1)-S(T)|&\le 2\cdot 2.5167\le 5.0334
\\
     |S(T+1)-S(T-1)|&\le 2\cdot 2.5167\le 5.0334.
    \end{aligned}
\end{equation*}
Substituting the estimates above into \eqref{eq:t} and \eqref{eq:t-1}, we can conclude that, for $T+1> 30610046000$, one has
\begin{equation*}
    \begin{aligned}
        &N(T+1)-N(T)\\&\le \left(\frac{1}{2\pi}+2C_1\right)\log T+2\min\{C_2\log\log T +\Tilde{C_3},C'_2\log\log T +\Tilde{C'_3}\}+\frac{3}{4\pi}-\frac{1}{2\pi}\log(2\pi e)+\frac{1}{25T},
    \end{aligned}
\end{equation*}
and
 \begin{equation*}
    \begin{aligned}
      &  N(T+1)-N(T-1)\\
      &\le \left(\frac{1}{\pi}+2C_1\right)\log T+\min\{2C_2\log\log T+\mathcal{D}_3,2C'_2\log\log T+\mathcal{D'}_3\}+\frac{1}{25(T-1)},
    \end{aligned}
\end{equation*}
where
\[
\mathcal{D}_3=2\Tilde{C_3}+\frac{\log 3-\log(2\pi e)}{\pi}
\quad\text{and}\quad
\mathcal{D'}_3=2\Tilde{C'_3}+\frac{\log 3-\log(2\pi e)}{\pi}.
\]
On the other hand, if  $3\le T+1\le 30610046000$, then we deduce
\begin{equation*}
    N(T+1)-N(T)\le \frac{1}{2\pi}\log T + 5.0334+\frac{3}{4\pi}-\frac{1}{2\pi}\log(2\pi e)+\frac{1}{25T}\le  \frac{1}{2\pi}\log T+4.8405
\end{equation*}
and 
\begin{equation*}
    N(T+1)-N(T-1)\le \frac{1}{\pi}\log T + 5.0334+\frac{\log 3}{\pi}-\frac{1}{\pi}\log(2\pi e)\le  \frac{1}{\pi}\log T+4.4798.
\end{equation*}

Now, it remains to find the lower bounds. Similarly to \eqref{upper1} and \eqref{upper2}, we have
\begin{equation*}
    \begin{aligned}
        &\frac{T+1}{2\pi} \log \left(\frac{T+1}{2 \pi e}\right)-\frac{T}{2\pi} \log \left(\frac{T}{2 \pi e}\right)>\frac{1}{2\pi}\log T+\frac{1}{2\pi}-\frac{1}{2\pi}\log(2\pi e),
        \end{aligned}
\end{equation*}
and
\begin{equation*}
    \begin{aligned}
        &\frac{T+1}{2\pi} \log \left(\frac{T+1}{2 \pi e}\right)-\frac{T-1}{2\pi} \log \left(\frac{T-1}{2 \pi e}\right)>\frac{1}{\pi}\log T+\frac{1}{\pi}+\frac{\log(3/4)}{2\pi}-\frac{1}{\pi}\log(2\pi e).
        \end{aligned}
\end{equation*}
It then follows that 
\begin{equation*}
    \begin{aligned}
        &N(T+1)-N(T-1)\\
        &>\frac{1}{\pi}\log T+\frac{1}{\pi}+\frac{\log(3/4)}{2\pi}-\frac{1}{\pi}\log(2\pi e)-|S(T+1)-S(T-1)|-\frac{1}{25(T-1)}\\
        &> \left(\frac{1}{\pi}-2.000001C_1\right)\log T-2\min\{C_2\log\log T +\mathcal{E},C'_2\log\log T +\mathcal{E'}\}-\frac{1}{25(T-1)},
    \end{aligned}
\end{equation*}
 for $T+1> 30610046000$, where
\[
\mathcal{E}=2\Tilde{C_3}+\frac{1}{\pi}+\frac{\log(3/4)}{2\pi}-\frac{1}{\pi}\log(2\pi e)\quad \text{and}\quad \mathcal{E'}=2\Tilde{C^{\prime}_3}+\frac{1}{\pi}+\frac{\log(3/4)}{2\pi}-\frac{1}{\pi}\log(2\pi e).
\]
Finally, if  $3\le T+1\le 30610046000$, then we have
\[
N(T+1)-N(T)>\frac{1}{2\pi}\log T+\frac{1}{2\pi}-\frac{1}{2\pi}\log(2\pi e)-5.0334-\frac{1}{25T}
\]
and
\[
N(T+1)-N(T-1)>\frac{1}{\pi}\log T+\frac{1}{\pi}+\frac{\log(3/4)}{2\pi}-\frac{1}{\pi}\log(2\pi e)-5.0334-\frac{1}{25(T-1)}.
\]
\begin{remark}
    The $0.000001$ goes inside the approximation of the last decimal place in the various constants. 
\end{remark}
\section*{Acknowledgments}
The authors are grateful to the referees for their careful reading of this article and their insightful comments.
The authors thank Andrew Fiori, Nathan Ng, and Tim Trudgian for the helpful discussions and comments. The authors are grateful to David Platt for the computation and comments that led to writing the subsection in the introduction. The authors would like to thank the technical staff at NCI, the University of Bristol ACRC and the University of New South Wales for allocating machine hours on Gadi through the UNSW Resource Allocation Scheme. The authors also thank PIMS-CRG ``$L$-functions in Analytic Number Theory'' for the support of attending the Comparative Prime Number Theory symposium (CPNTS). They would like to express their gratitude to all the CPNTS organisers, including Lucile Devin, Daniel Fiorilli, Alia Hamieh, Habiba Kadiri, Wanlin Li, Greg Martin, and Nathan Ng, for creating wonderful collaboration opportunities that enabled the start of this project. The first author also thanks the organisers of the workshop ``Analytic and Explicit results of zeros of $ L$-functions" (B{\c e}dlewo, September 23-27, 2024), where helpful discussions with David Platt took place. 
%\clearpage
\printbibliography

@article{FK15,
AUTHOR={Faber, L. and Kadiri, H.},
     TITLE = {New bounds for {$\psi(x)$}},
   JOURNAL = {Math. Comp.},
  FJOURNAL = {Math. Comp.},
    VOLUME = {84},
      YEAR = {2015},
    NUMBER = {293},
     PAGES = {1339--1357},
      ISSN = {0025-5718},
   MRCLASS = {11M06 (11M26)},
  MRNUMBER = {3315511},
MRREVIEWER = {Yuanyou Furui Cheng},
       DOI = {10.1090/S0025-5718-2014-02886-X},
       URL = {https://doi.org/10.1090/S0025-5718-2014-02886-X},
}

@article{Trudgian20121053,
	author = {Trudgian,  T. S.},
	title = {An improved upper bound for the argument of the Riemann zeta-function on the critical line},
	year = {2012},
	journal = {Math. Comp.},
	volume = {81},
	number = {278},
	pages = {1053 – 1061},
}

@article{Tr14-2,
    AUTHOR = {Trudgian, T. S.},
     TITLE = {An improved upper bound for the argument of the {R}iemann
              zeta-function on the critical line {II}},
   JOURNAL = {J. Number Theory},
  FJOURNAL = {Journal of Number Theory},
    VOLUME = {134},
      YEAR = {2014},
     PAGES = {280--292},
      ISSN = {0022-314X},
   MRCLASS = {11M06 (11M26)},
  MRNUMBER = {3111568},
MRREVIEWER = {Haseo Ki},
       DOI = {10.1016/j.jnt.2013.07.017},
       URL = {https://doi.org/10.1016/j.jnt.2013.07.017},
}

@article{Ro41,
    AUTHOR = {Rosser, J. B.},
     TITLE = {Explicit bounds for some functions of prime numbers},
   JOURNAL = {Amer. J. Math.},
  FJOURNAL = {American Journal of Mathematics},
    VOLUME = {63},
      YEAR = {1941},
     PAGES = {211--232},
      ISSN = {0002-9327},
   MRCLASS = {10.0X},
  MRNUMBER = {3018},
MRREVIEWER = {R. D. James},
       DOI = {10.2307/2371291},
       URL = {https://doi.org/10.2307/2371291},
}

@article{Ba18,
AUTHOR={Backlund, R. J.},
     TITLE = {\"{U}ber die {N}ullstellen der {R}iemannschen {Z}etafunktion},
   JOURNAL = {Acta Math.},
  FJOURNAL = {Acta Mathematica},
    VOLUME = {41},
      YEAR = {1916},
    NUMBER = {1},
     PAGES = {345--375},
      ISSN = {0001-5962},
   MRCLASS = {DML},
  MRNUMBER = {1555156},
       DOI = {10.1007/BF02422950},
       URL = {https://doi.org/10.1007/BF02422950},
}

@article{Gr13, 
AUTHOR={Grossmann, J.},
TITLE={\"Uber die Nullstellen der Riemannschen Zeta-Funktion und der Dirichletschen $L$-Funktionen},
note={PhD thesis, Georg-August-Universit\"at G\"ottingen},
year={1913},
}

@article{vMo05,
    AUTHOR = {Von Mangoldt, H. C. F.},
     TITLE = {Zur {V}erteilung der {N}ullstellen der {R}iemannschen
              {F}unktion {$\xi(t)$}},
   JOURNAL = {Math. Ann.},
  FJOURNAL = {Mathematische Annalen},
    VOLUME = {60},
      YEAR = {1905},
    NUMBER = {1},
     PAGES = {1--19},
      ISSN = {0025-5831},
   MRCLASS = {DML},
  MRNUMBER = {1511287},
       DOI = {10.1007/BF01447494},
       URL = {https://doi.org/10.1007/BF01447494},
}

@article{HASANALIZADE2022219,
title = {Counting zeros of the Riemann zeta function},
journal = {J. Number Theory},
volume = {235},
pages = {219-241},
year = {2022},
author = {Elchin Hasanalizade and Quanli Shen and Peng-Jie Wong},
}

@article{patel2023explicit,
      title={An explicit sub-Weyl bound for $\zeta(1/2 + it)$}, 
      author={Dhir Patel and Andrew Yang},
      year={2024},
journal={J. Number Theory},
note={To appear },
}

@article{hiary_explicit_2016,
	title = {An explicit van der {Corput} estimate for  $\zeta( 1 / 2 + i t )$},
	volume = {27},
	number = {2},
	journal = {Indag. Math.},
	author = {Hiary, Ghaith A.},
	year = {2016},
	pages = {524--533},
}

@article{hiary_improved_2024,
	title = {An improved explicit estimate for $\zeta(1/2 + it)$},
	volume = {256},
	journal = {J. Number Theory},
	author = {Hiary, Ghaith A. and Patel, Dhir and Yang, Andrew},
	year = {2024},
	pages = {195--217},
}

@article{hiary2023explicit,
      title={Explicit bounds for the Riemann zeta-function on the 1-line}, 
      author={Ghaith A. Hiary and Nicol Leong and Andrew Yang},
      journal={Funct. Approx. Comment. Math.},
      year={2024},
      note={(to appear)},
}

@article{BELLOTTI2024128249,
title = {Explicit bounds for the Riemann zeta function and a new zero-free region},
journal = {J. Math. Anal. Appl.},
volume = {536},
number = {2},
pages = {128249},
year = {2024},
author = {Chiara Bellotti},
}

@article{patel_explicit_2022,
	title = {An explicit upper bound for $\zeta( 1 + i t )$ },
	volume = {33},
	number = {5},
	journal = {Indag. Math. },
	author = {Patel, Dhir},
	year = {2022},
	pages = {1012--1032},
}

@article{YANG2024128124,
title = {Explicit bounds on $\zeta(s)$ in the critical strip and a zero-free region},
journal = {J. Math. Anal. Appl.},
volume = {534},
number = {2},
pages = {128124},
year = {2024},
author = {Andrew Yang},
}

@article{leong2024explicitestimateslogarithmicderivative,
      title={Explicit estimates for the logarithmic derivative and the reciprocal of the Riemann zeta-function}, 
      author={Nicol Leong},
      year={2024},
      note={ArXiv:2405.04869},
}

@article{Hasanalizade2021CountingZO,
  title={Counting zeros of Dedekind zeta functions},
  author={Elchin Hasanalizade and Quanli Shen and PENG-JIE Wong},
  journal={Math. Comp.},
  year={2021},
  volume={91},
  pages={277-293},
}

@article{bennett_counting_2021,
	title = {Counting zeros of {Dirichlet} $L$-functions},
	volume = {90},
	number = {329},
	journal = {Math. Comp.},
	author = {Michael A. Bennett  and Greg Martin and Kevin O’Bryant  and Andrew Rechnitzer},
	year = {2021},
	pages = {1455--1482},
}

@article{dudekcubes,
author = {Adrian W. Dudek},
title = {{An explicit result for primes between cubes}},
volume = {55},
journal = {Funct. Approx. Comment. Math.},
number = {2},
pages = {177 -- 197},
year = {2016},
}

@article{michaela2024errortermexplicitformula,
      title={On the error term in the explicit formula of Riemann-von Mangoldt II}, 
      author={Daniel R. Johnston and Michaela Cully-Hugill},
      journal={Funct. Approx. Comment. Math. },
      year={2024},
      note={To appear},
}

@online{Bellotti_argument_zeta ,
author = {Bellotti, Chiara},
title = {Argument and zero counting function estimates for zeta},
note={GitHub repository available at},
url={https://github.com/ChiaraBellotti/Argument-and-counting-function-estimates-for-Riemann-zeta-function},
}

@online{LMFDB,
shorthand    = {LMFDB},
  author       = {{The LMFDB Collaboration}},
  title        = {The {L}-functions and modular forms database},
url={https://www.lmfdb.org/},
}

@article{PlaTru21RH,
author = {Platt, D. J. and Trudgian,  T. S.},
title = {The Riemann hypothesis is true up to $3\cdot 10^{12}$},
journal = {Bull. Lond. Math. Soc.},
volume = {53},
number = {3},
pages = {792-797},
doi = {https://doi.org/10.1112/blms.12460},
url = {https://londmathsoc.onlinelibrary.wiley.com/doi/abs/10.1112/blms.12460},
eprint = {https://londmathsoc.onlinelibrary.wiley.com/doi/pdf/10.1112/blms.12460},
year = {2021}
}

@article{Pla17,
author = {Platt, D. J.},
year = {2016},
month = {07},
pages = {1},
title = {Isolating some non-trivial zeros of zeta},
volume = {86},
journal = {Math. Comp.},
doi = {10.1090/mcom/3198}
}

@article{FGH2005,
author = {Farmer, David and Gonek, S. and Hughes, Christopher},
year = {2005},
month = {07},
pages = {},
title = {The maximum size of $L$-functions},
volume = {2007},
journal = {Journal für die reine und angewandte Mathematik (Crelles Journal)},
doi = {10.1515/CRELLE.2007.064}
}

@article{PLATT2015842,
title = {An improved explicit bound on $|\zeta(1/2+it)|$},
journal = {J. Number Theory},
volume = {147},
pages = {842-851},
year = {2015},
author = {Platt, D. J. and Trudgian,  T. S.}
}

@incollection{AST_1987__147-148__325_0,
     author = {Pintz, J.},
     title = {An effective disproof of the {Mertens} conjecture},
     booktitle = {Journ\'ees arithm\'etiques de Besan\c{c}on},
     series = {Ast\'erisque},
     pages = {325--333},
     publisher = {Soci\'et\'e math\'ematique de France},
     number = {147-148},
     year = {1987},
     zbl = {0623.10031},
}

@article{rozmarynowycz202,
      title={A new upper bound on the smallest counterexample to the Mertens conjecture}, 
      author={John Rozmarynowycz and Seungki Kim},
      year={2023},
      note={ArXiv:2305.00345},
}

@article{qingyi2024,
      title={Explicit Bound of $|\zeta\left(1+it\right)|$}, 
      author={Eunice Hoo Qingyi and Lee-Peng Teo},
      year={2024},
      note={ArXiv:2412.00766},
}

@article{simonic,
title = {On explicit estimates for $S(t)$, $S_1(t)$, and $\zeta(1/2+it)$ under the Riemann Hypothesis},
journal = {Journal of Number Theory},
volume = {231},
pages = {464-491},
year = {2022},
author = {Aleksander Simoni\v{c}},
}

@article{CCM2013,
  title={Bounding and on the Riemann hypothesis},
  author={Carneiro, Emanuel and Chandee, Vorrapan and Milinovich, Micah B},
  journal={Mathematische Annalen},
  volume={356},
  number={3},
  pages={939--968},
  year={2013},
  publisher={Springer}
}

@article{CCM2019,
 ISSN = {02141493, 20144350},
 URL = {https://www.jstor.org/stable/26860751},
 author = {Emanuel Carneiro and Andrés Chirre and Micah B. Milinovich},
 journal = {Publicacions Matemàtiques},
 number = {2},
 pages = {pp. 601--661},
 publisher = {Universitat Autònoma de Barcelona},
 title = {Bandlimited approximations and estimates for the Riemann zeta-function},
 urldate = {2025-06-10},
 volume = {63},
 year = {2019}
}

@article{fiori2025notephragmenlindeloftheorem,
      title={A Note on the Phragmen-Lindelof Theorem}, 
      author={Andrew Fiori},
      year={2025},
      note={ArXiv:2502.13282},
}

@article{Boberhiary,
author = {Jonathan W. Bober and Ghaith A. Hiary},
title = {New Computations of the Riemann Zeta Function on the Critical Line},
journal = {Experimental Mathematics},
volume = {27},
number = {2},
pages = {125--137},
year = {2018},
publisher = {Taylor \& Francis},
}

@article{ARB,
author = {Johansson, Fredrik},
title = {Arb: Efficient Arbitrary-Precision Midpoint-Radius Interval Arithmetic},
year = {2017},
issue_date = {Aug. 2017},
publisher = {IEEE Computer Society},
address = {USA},
volume = {66},
number = {8},
journal = {IEEE Trans. Comput.},
month = {aug},
pages = {1281–1292},
numpages = {12}
}

@phdthesis{plattthesis,
    author = {Platt, D. J.},
    title = {Computing degree 1 $L$-functions rigorously},
    school =  {University of Bristol},
    year = {2011},
}

@article{platt_isolating_2017,
	title = {Isolating some non-trivial zeros of zeta},
	volume = {86},
	number = {307},
	journal = {Mathematics of Computation},
	author = {Platt, D. J.},
	year = {2017},
	pages = {2449--2467},
}

@online{NCI,
shorthand    = {NCI},
  title        = {NCI HPC Systems},
url={https://nci.org.au/our-systems/hpc-systems},
}

@online{ACRC,
shorthand    = {BCUG},
  title        = {BlueCrystal User Guide},
url={https://www.acrc.bris.ac.uk/pdf/bc-user-guide.pdf},
}
\iftrue

\newpage
\appendix

\section{A Numerical Study}
\begin{center}
    Andrew Fiori \footnote{
    Department of Mathematics and Statistics, University of Lethbridge,
4401 University Drive,
Lethbridge, Alberta,
T1K 3M4,
Canada
    \\Andrew Fiori's Research is supported by NSERC Discovery Grant RGPIN-2020-05316.}
\end{center}

Under the Riemann Hypothesis one expects a bound of the form:
\begin{equation} \label{expbound} \abs{N(T) - \Big(\frac{T}{2\pi}\log \frac{T}{2\pi e} + 7/8 \Big)} \leq C\frac{\log{T}}{\log{\log{T}}} + O\left(\frac{1}{T}\right). \end{equation}
Explicit values for $C$ have been computed in \cite{simonic}. Moreover, it is known that $C\leq \frac{1}{4} + \epsilon$ for all $\epsilon>0$ by \cite{CCM2013}. This was refined by \cite{CCM2019} who established $C\leq\frac{1}{4}+O(\frac{1}{\log\log T})$.
However, it is conjectured in \cite{FGH2005} that the actual size is:
\begin{equation} \label{expbound2} \abs{N(T) - \Big(\frac{T}{2\pi}\log \frac{T}{2\pi e} + 7/8 \Big)} < \left(\frac{1}{\sqrt{2}\pi} + o(1)\right)\sqrt{\log(T)\log\log(T)} . \end{equation}

One goal of this appendix is to provide numerical evidence for this stronger conjecture.
A second goal is to provide useful input for future works to improve upon effective bounds on $N(T)$ by summarizing what is known using the work of  \cite{Pla17}. 
We provide such a result in Theorem \ref{thetheorem}.

Here we study $T<30\,610\,046\,000$ by using the list of the zeros for the zeta function as computed by \cite{Pla17} and made available at \cite{LMFDB}.
This database of the first $103\,800\,788\,359$ many zeros is broken up into 14\,580 intervals. For practical reasons we analyze the data using these intervals.
We note that although \cite{PlaTru21RH} has verified RH to $3\cdot 10^{12}$, they have not produced a database of zeros.

We shall denote the imaginary part of the $n$th zero, ordered by height $t_n$. Because all the zeros up to $30\,610\,046\,000$ are simple we know that in the interval under consideration $N(t_n) = n$.

\subsection{Average Values}

\begin{obs}
On each of the intervals of zeros produced by \cite{Pla17}, the average value of the function 
\[ N(t_n) -  \frac{t_n}{2\pi}\log \frac{t_n}{2\pi e} \]
evaluated at the zeros, $t_n$, of zeta in that interval is approximately $11/8$.\\
Excluding the first two intervals, where the deviation from the average is respectively $-5.07685\cdot 10^{-05}$ and $2.48229 \cdot 10^{-05}$,
on each of the intervals the deviation of these average values from $11/8$ is bounded by $5.6678910^{-07}$.
\end{obs}

\subsection{Range of Values}

\begin{proposition}\label{atzeros}
On any interval beginning and ending at zeros of $\zeta$ the maximum value of 
\[  \epsilon^{+}(T) =  N(T) - \Big(\frac{T}{2\pi}\log \frac{T}{2\pi e} + 11/8\Big)  -  \frac{1}{\sqrt{2}\pi}\sqrt{\log(T)\log\log(T)}  \]
will be taken at $t_n$, the exact ordinate of a zero of $\zeta$.
Moreover, if all zeros on the interval are simple, then the infimum of
\[   \epsilon^{-}(T) = N(T) -  \Big(\frac{T}{2\pi}\log \frac{T}{2\pi e} + 11/8\Big)  + \frac{1}{\sqrt{2}\pi}\sqrt{\log(T)\log\log(T)} \]
will be exactly $1$ less than the value $\epsilon^{-}(T)$  takes on at $t_n$, the exact ordinate of a zero of $\zeta$.
\end{proposition}
\begin{proof}
Notice that the respective functions $ \epsilon^{\pm}(T)$ are continuous and decreasing on any interval $[t_n, t_{n+1})$ between consecutive zeros and that
\[ \lim_{t\rightarrow t_{n+1}^-}  \epsilon^{-}(t) = \epsilon^{-}(t_{n+1}) - 1. \]
The results follow immediately.
\end{proof}

As a result of the above, we focus our attention on the study of minimum values of $\epsilon^{-}(t_n)$ and maximum values of $\epsilon^{+}(t_n)$ at zeta zeros.

\begin{obs}
The maximum value of  $\epsilon^{+}(t_n)$ in the database of zeros produced by \cite{Pla17} is approximately $0.0920937$ and occurs at $n=2953463649$ with $t_n$ approximately $1035537870.14791389$. 
The minimum value of  $\epsilon^{-}(t_n)$ in the database of zeros produced by \cite{Pla17} is approximately $-0.0827069$ and occurs at $n=48227304665$ with $t_n$ approximately $14727556977.25899340$.
\end{obs}

\begin{theorem}\label{thetheorem}
For $e <T<30\,610\,046\,000$ we have
\[             -\frac{\sqrt{\log(T)\log\log(T)}}{\sqrt{2}\pi} - 1 -  0.082707         <          N(T) -  \frac{T}{2\pi}\log \frac{T}{2\pi e} - \frac{11}{8}  <  \frac{\sqrt{\log(T)\log\log(T)}}{\sqrt{2}\pi}  +  0.092094. \]
\end{theorem}
\begin{proof}
This is an immediate consequence of the previous observation and proposition.
\end{proof}

\subsection{Extreme Values are Rare}

The following observation says that in some sense extreme values are rare.
\begin{obs}
The function $\epsilon^{+}(t_n)$ is almost always negative, indeed, it is negative at all zeros with $0<t<30,610,046,000$ except for those listed in Table \ref{tabextreme}.
Similarly, the function $\epsilon^{-}(t_n)$ is almost always positive, indeed, it is positive at all zeros with $0<t<30,610,046,000$  for those listed in Table \ref{tabextreme}. We notice that the frequency of the extreme values is decreasing as both $n$ and $t_n$ increase. We also note that there are no examples of consecutive extreme values.

\begin{table}[h]
\caption{List of all positive (respectively negative) values of $\epsilon^{+}(t_n)$ (respectively $\epsilon^{-}(t_n)$) for $0<t<30,610,046,000$} \label{tabextreme}
\centering
{\small
\begin{tabular}{rrl}
$n$ \qquad& $t_n$\qquad\qquad &\quad  $\epsilon^{+}(t_n)$\\
\hline
7330779 & 3745331.534911 & 0.0045727 \\
10014001 & 4998855.443421 & 0.0228842 \\
30930930 & 14253736.600191 & 0.0215612 \\
106941331 & 45420475.080263 & 0.0548687 \\
121934174 & 51361501.783167 & 0.0633788 \\
342331986 & 135399343.427052 & 0.0852310 \\
486250460 & 188404036.065583 & 0.0077204 \\
1333195695 & 487931556.151002 & 0.0711065 \\
1819794287 & 654800601.959837 & 0.0016382 \\
2953463649 & 1035537870.147914 & 0.0920937 \\
4711070126 & 1611978781.026883 & 0.0552043 \\
6020412879 & 2034221491.431262 & 0.0040533 \\
6276413932 & 2116223525.742432 & 0.0136917 \\
6916958115 & 2320709265.610272 & 0.0183240 \\
7895552868 & 2631384288.230762 & 0.0134043 \\
18019870103 & 5765666759.059866 & 0.0101465 \\
29425625937 & 9196418366.325099 & 0.0141895 \\
31587712923 & 9839079152.616086 & 0.0169750 \\
43668302178 & 13396993184.932842 & 0.0314554 \\
44121363503 & 13529486654.222228 & 0.0049450 \\
71876944166 & 21550885080.446041 & 0.0076388 \\
100093914039 & 29565113205.570534 & 0.0184384 \\
\\\\\\\\\\\\\\\\\\\\
\end{tabular}\;\;\begin{tabular}{rrl}
$n$ \qquad& $t_n$\qquad\qquad &\quad  $\epsilon^{-}(t_n)$\\
\hline
337917 & 223936.368134 & -0.0206077 \\
2009961 & 1137116.070608 & -0.0268423 \\
10869861 & 5393528.443012 & -0.0021561 \\
13999527 & 6820051.890986 & -0.0219395 \\
37592217 & 17095484.271828 & -0.0359387 \\
83088045 & 35862210.311523 & -0.0463096 \\
88600097 & 38084045.549954 & -0.0491801 \\
141617808 & 59096901.323297 & -0.0082036 \\
164689303 & 68084444.336913 & -0.0322461 \\
191297537 & 78359876.488247 & -0.0148321 \\
225291159 & 91369499.494965 & -0.0092099 \\
566415149 & 217536164.326180 & -0.0121163 \\
1081300142 & 400354486.072002 & -0.0593335 \\
1257893678 & 461849910.598599 & -0.0262264 \\
1372703319 & 501584522.950737 & -0.0002137 \\
1955876862 & 701027396.312615 & -0.0096394 \\
2305634166 & 819113670.556185 & -0.0026561 \\
5134032906 & 1748936581.577121 & -0.0142280 \\
5136505385 & 1749735519.272913 & -0.0508501 \\
8864769308 & 2937266043.546390 & -0.0418297 \\
9430966584 & 3115208316.829027 & -0.0490178 \\
9532704476 & 3147127461.727906 & -0.0055879 \\
18629248201 & 5951053644.636571 & -0.0008106 \\
19859326408 & 6324431638.934122 & -0.0452773 \\
21082098810 & 6694540279.310641 & -0.0362361 \\
22909699222 & 7245905144.854708 & -0.0030872 \\
48227304665 & 14727556977.258993 & -0.0827069 \\
77728515578 & 23222574401.823281 & -0.0515078 \\
86585440777 & 25742609309.393488 & -0.0073926 \\
87198634344 & 25916653877.755976 & -0.0017881 \\
97495263831 & 28831591819.434777 & -0.0222335 \\
103274388030 & 30461757456.864450 & -0.0421276 \\
\end{tabular}
}
\end{table}
\end{obs}

\begin{conjecture}
One could make a variety of conjectures of different strengths based on Table \ref{tabextreme}.
\begin{enumerate}
\item The number of extreme values up to height $T$ is less than $C\log T$ for some constant $C$.
\item The set of $n$ for which these functions are respectively positive or negative has natural density zero.
\end{enumerate}
\end{conjecture}

\subsection{Extreme Values are Common}

The following observation says that in some sense extreme values are common.
\begin{obs}
If for each interval of zeros produced by \cite{Pla17}, we compute the maximum value of  $\epsilon^{+}(t_n)$ on that interval, then the minimum of these values is $-0.362463$.
Similarly, if for each interval of zeros produced by \cite{Pla17} we compute the minimum value of  $\epsilon^{-}(t_n)$  on that interval, then the maximum of these values is $0.361383$.\\
For context note that $\frac{1}{\sqrt{2}\pi}\sqrt{\log(3\cdot10^{10})\log\log(3\cdot10^{10})}$ is approximately $1.97241$ and $2\pi/\log(3\cdot10^{10}/(2\pi e))$ is approximately $0.295171$.
In particular, for each interval the function 
\[  N(t_n) -  \frac{t_n}{2\pi}\log \frac{t_n}{2\pi e} -11/8  \]
takes on values relatively close to both the theoretical upper and lower extremes under consideration.
\end{obs}

\subsection{Clusters of Zeros}

One common use of estimates on $N(t)$ is to bound the number of zeros on an interval.
The quality of the bounds one obtains using bounds on $N(t)$ improves with $t$.
Consequently, it is useful to have explicit bounds on the number of zeros in intervals for small $t$.

\begin{obs}
For $e<t<30\;610\;045\;999$, the maximum value for $\frac{N(t+1)-N(t-1) }{\log t}$  happens around $t=2261.88$ where $N(t+1)-N(t-1)=4$.
Consequently, on this interval
\[  N(t+1)-N(t-1) < 0.517869686 \log t . \]
Moreover, for each $n$ the smallest value of $t$ for which $N(t+1)-N(t-1)=n$ is given in Table \ref{tabNTp1m1}, from which one may bound $N(t+1)-N(t-1)$ with a step function, or any other function which happens to exceed that step function.
\begin{table}[h]
\caption{Minimum value of $t$ with $N(t+1)-N(t-1)=n$}
\label{tabNTp1m1}
\begin{tabular}{c|c}
n& t \\
\hline
1&  13.1347251417346937904572\\
2&  48.7738324776723021819167\\
3&  356.952685101632273755128 \\
4& 2261.87830538116111223015 \\
5& 27134.3628475733906424560\\
6& 221227.766664702101313669\\
7& 2603074.61468824424587333\\
8& 21297085.9439615105210553\\
9& 254721517.418748602610351 \\
10& 2786055796.5252751861828\\
11& 29731208527.9429140229012\\
12& larger than 30610045999
\end{tabular}
\end{table}
\end{obs}

\end{document}